\documentclass[reqno,final]{tac}
\pdfoutput=1

\usepackage{macros}
\usepackage{hyperref}
\hypersetup{final}

\hypersetup{
  bookmarks=true,
  colorlinks=true,
  linkcolor=black,
  citecolor=black,
  filecolor=black,
  urlcolor=black,
  pdftitle={On Finitary Functors},
  pdfauthor={Ji\v{r}\'\i\ Ad\'amek, Stefan Milius, Lurdes Sousa, and
    Thorsten Wi\ss\/mann},
  pdfkeywords={coalgebras, recursive, well-founded},
  pdfduplex={DuplexFlipLongEdge},
}

\usepackage[notref,notcite]{showkeys}
\usepackage{xcolor}
\usepackage{soul}
\usepackage{seqsplit}
\usepackage{xstring}

\makeatletter
\renewcommand*\showkeyslabelformat[1]{%
\noexpandarg%
\StrSubstitute{#1}{ }{\textvisiblespace}[\TEMP]%
\begin{minipage}[t]{\marginparwidth}%
  \normalfont\small\ttfamily\(\{\){\color{red!50!black}\expandafter\seqsplit\expandafter{\TEMP}}\(\}\)\end{minipage}%
}
\makeatother

\usepackage[marginclue,nomargin,footnote]{fixme}

\newcommand{\takeout}[1]{\empty}
\title{On Finitary Functors}
\author{J.~Ad\'{a}mek, S.~Milius, L.~Sousa and T.~Wi{\ss}mann}
\thanks{J.~Ad\'amek was supported by the Grant Agency of the Czech Republic under the grant 19-009025. \\ S. Milius and T. Wi\ss\/mann acknowledge support by the Deutsche Forschungsgemeinschaft (DFG) under project MI~717/5-2.
\\L.~Sousa was partially supported by the Centre for Mathematics of the
University of Coimbra -- UID/MAT/00324/2019, funded by the Portuguese
Government through FCT/MEC and co-funded by the European Regional Development Fund through the Partnership Agreement PT2020.}
\address{%
  Department of Mathematics, Faculty of Electrical Engineering, Czech Technical University in Prague, Czech Republic\\[5pt]
  Lehrstuhl f\"ur Informatik 8 (Theoretische Informatik), Friedrich-Alexander-Universit\"at Er\-lan\-gen-N\"urn\-berg, Germany \\[5pt]
  CMUC, University of Coimbra,  Portugal \&  ESTGV, Polytechnic Institute of Viseu, Portugal\\[5pt]
}
\eaddress{j.adamek@tu-braunschweig.de\CR mail@stefan-milius.eu\CR
  sousa@estv.ipv.pt\CR thorsten.wissmann@fau.de}

\copyrightyear{2019}

\date{\today}

\begin{document}
%
%
\FXRegisterAuthor{sm}{asm}{SM}
\FXRegisterAuthor{ja}{aja}{JA}
\FXRegisterAuthor{ls}{als}{LS}
\FXRegisterAuthor{tw}{twm}{TW}

\maketitle

\begin{abstract}
  A simple criterion for a functor to be finitary is presented: we call $F$
  finitely bounded if for all objects $X$ every finitely generated subobject of
  $FX$ factorizes through the $F$-image of a finitely generated subobject of
  $X$. This is equivalent to $F$ being finitary for all functors between
  `reasonable' locally finitely presentable categories, provided that $F$
  preserves monomorphisms. We also discuss the question when that last
  assumption can be dropped.
  The answer is affirmative for functors between categories such as Set,
  K-Vec (vector spaces), boolean algebras, and actions of any finite group
  either on Set or on K-Vec for fields K of characteristic 0.

  All this generalizes to locally $\lambda$-presentable categories,
  $\lambda$-accessible functors and $\lambda$-presentable algebras. As
  an application we obtain an easy proof that the Hausdorff functor on
  the category of complete metric spaces is $\aleph_1$-accessible.
\end{abstract}

\section{Introduction}
\label{sec:intro}

In a number of applications of categorical algebra, {\em finitary
  functors}, i.e.~functors preserving filtered colimits, play an
important role. For example, the classical varieties are precisely the
categories of algebras for finitary monads over $\Set$. How does one
recognize that a functor $F$ is finitary? For endofunctors of $\Set$
there is a simple necessary and sufficient condition: given a set $X$,
every finite subset of $FX$ factorizes through the image by $F$ of a
finite subset of $X$. This condition can be formulated for general
functors $F\colon \A\to \B$:
given an object $X$ of $\A$, every finitely generated subobject of $FX$ in $\B$ is required to factorize through the image by $F$  of a finitely generated subobject of $X$ in $\A$. We call such functors {\em finitely bounded}. For functors between locally finitely presentable categories which preserve monomorphisms we prove
$$\text{finitary} \iff \text{finitely bounded}$$
whenever finitely generated objects are finitely presentable. (The last condition is, in fact, not only sufficient but also necessary for the above equivalence.)

What about general functors, not necessarily preserving monomorphisms? We prove
the above equivalence whenever $\A$ is a strictly locally finitely presentable
category, see Definition \ref{D:strict}. Examples of such categories are sets,
vector spaces, group actions of finite groups, and $S$-sorted sets with $S$
finite. Conversely, if the above equivalence is true for all functors from $\A$
to $\Set$, we prove that a weaker form of
strictness holds for $\A$.

All of the above results can be also formulated for locally
$\lambda$-presentable categories and $\lambda$-accessible functors.
We use this to provide a
simple proof that the Hausdorff functor on the category of complete metric
spaces is countably accessible.

\paragraph{Acknowledgement.} 
We are very grateful to the anonymous referee: he/she
found a substantial simplification of the main definition (strictly and semi-strictly lfp
category) and pointed us to atomic toposes (see
Example \ref{E:strict}\ref{e:topos}).

We are also grateful for discussions about pure subobjects with John Bourke, Ivan
Di Liberti, and Ji\v{r}\'\i~Rosick\'y.

\section{Preliminaries}\label{sec:prel}

In this section we present properties of  finitely presentable and finitely generated objects which will be useful in the subsequent sections.

 Recall that an
object $A$ in a category $\A$ is called \emph{finitely presentable} if its
hom-functor $\A(A,-)$ preserves filtered colimits, and $A$ is called
\emph{finitely generated} if $\A(A,-)$ preserves filtered colimits of
monomorphisms -- more precisely, colimits of filtered diagrams $D\colon \D
\to \A$ for which $Dh$ is a monomorphism in $\A$ for every morphism
$h$ of $\D$.  

\begin{notation}
  For a category $\A$ we denote by
  \[
    \Afp \qquad\text{and}\qquad\Afg
  \]
  full subcategories of $\A$ representing (up to isomorphism) all finitely
  presentable and finitely generated objects, respectively.

  Subobjects $m\colon M \monoto A$ with $M$ finitely generated are called
  \emph{finitely generated subobjects}.
\end{notation}

Recall that $\A$ is a \emph{locally finitely presentable} category, shortly
\emph{lfp} category, if it is cocomplete, $\Afp$ is small, and every
object is a colimit of a filtered diagram in $\Afp$.

We now recall a number of standard facts about
lfp categories~\cite{AdamekR94}.

\begin{remark}\label{R:prelim}
  Let $\A$ be an lfp category.
  \begin{enumerate}
  \item \label{I:factorization} By~\cite[Proposition~1.61]{AdamekR94}, $\A$ has
    (strong epi, mono)-factorizations of morphisms.

  \item \label{I:canColim} By~\cite[Proposition~1.57]{AdamekR94}, every object
    $A$ of $\A$ is the colimit of its \emph{canonical filtered diagram}
    \[
      D_A\colon \Afp/A \to \A \qquad (P \xrightarrow{p} A) \mapsto P, 
    \]
    with colimit injections given by the $p$'s.

  \item By~\cite[Theorem~2.26]{AdamekR94}, $\A$ is a free completion
    of $\Afp$ under filtered colimits. That is, for every functor
    $H\colon \Afp \to \B$, where $\B$ has filtered colimits, there is an
    (essentially unique) extension of $H$ to a finitary functor
    $\bar{H}\colon\A\to\B $. Moreover, this extensions can be formed as
    follows: for every object $A\in \A$ put
    \[
      \bar{H} A = \colim H\cdot D_A.
    \]

  \item \label{I:colimono} By~\cite[Proposition~1.62]{AdamekR94}, a colimit of a filtered
    diagram of monomorphisms has monomorphisms as colimit injections.
    Moreover, for every compatible cocone formed by monomorphisms, the
    unique induced morphism from the colimit is a monomorphism too.

  \item \label{I:fingen} By~\cite[Proposition~1.69]{AdamekR94}, an object $A$ is
    finitely generated iff it is a strong quotient of a finitely
    presentable object, i.e.~there exists a finitely presentable
    object $A_0$ and a strong epimorphism $e\colon A_0 \epito A$.
  \item It is easy to verify that every split quotient of a finitely
    presentable object is finitely presentable again.
  \end{enumerate}
\end{remark}

\begin{lemma}\label{L:union} Let $\A$ be an lfp category. A cocone of monomorphisms
  $c_i\colon Di\monoto C\; (i\in I)$ of a filtered diagram $D$ of
  monomorphisms is a colimit of $D$ iff it is a {\em union}; that is,
  iff $\id_C$ is the supremum of the subobjects $c_i\colon Di\monoto C$.
\end{lemma} 

\begin{proof} The `only if' direction is clear. For the `if' direction
  suppose that $c_i\colon  Di \monoto C$ have the union $C$, and let
  $\ell_i\colon  Di \to L$ be the colimit of $D$. Then, since $c_i$ is a
  cocone of $D$, we get a unique morphism $m\colon  L \to C$ with
  $m \cdot \ell_i = c_i$ for every $i$. By
  Remark~\ref{R:prelim}\ref{I:colimono}, all the $\ell_i$ and $m$ are
  monomorphisms, hence $m$ is a subobject of $C$. Moreover, we have
  that $c_i \leq m$, for every $i$. Consequently, since $C$ is the
  union of all $c_i$, $L$ must be isomorphic to $C$ via $m$, because
  $id_C$ is the largest subobject of $C$. Thus, the original cocone
  $c_i$ is a colimit cocone.
\end{proof}

\begin{remark} \label{R:refl}
  Colimits of filtered diagrams $D\colon \D\to \Set$ are precisely those cocones
  $c_i\colon D_i\to C$ ($i\in \obj \D$) of $D$ that have the following
  properties:
  \begin{enumerate}
  \item $(c_i)$ is jointly surjective, i.e.~$C=\bigcup c_i[D_i]$,  and
  \item given $i$ and elements $x,y\in D_i$ merged by $c_i$, then they are also
    merged by a connecting morphism $D_i\to D_j$ of $D$.
  \end{enumerate}
  This is easy to see: for every cocone $c_i'\colon D_i\to C'$ of $D$ define
  $f\colon C\to
  C'$ by choosing for every $x\in C$ some $y\in D_i$ with $x=c_i(y)$ and putting
  $f(x) = c_i'(y)$. By the two properties above, this is well defined and is unique
  with $f\cdot c_i = c_i'$ for all $i$.
\end{remark}

\begin{lemma}[Finitely presentable objects collectively reflect
  filtered colimits.]\label{L:refl}\\ 
  Let $\A$ be an lfp category and $D\colon \D\to \A$ a filtered
  diagram with objects $D_i$ ($i\in I$). A cocone $c_i\colon D_i\to C$
  of $D$ is a colimit iff for every $A\in \Afp$ the cocone
  \[
    c_i\cdot (-)\colon \A(A,D_i) \longrightarrow  \A(A,C)
  \]
  is a colimit of the diagram $\A(A,D-)$ in $\Set$.
\end{lemma}
Explicitly, the above property of the cocone $(c_i)$ states that for every
morphism $f\colon A\to C$ where $A\in \A_\fp$
\begin{enumerate}
\item a factorization through some $c_i$ exists, and
\item given two factorizations $f=c_i\cdot q_k$ for $k=1,2$, then
  $q_1,q_2\colon A\to D_i$ are merged by a connecting morphism of $\D$. The
  proof that this describes $\colim \A(A,D-)$ follows from Remark \ref{R:refl}.
\end{enumerate}
\begin{proof}
  If $(c_i)$ is a colimit, then since $\A(A,-)$ preserves filtered colimits,
  the cocone of all $\A(A,c_i)=c_i\cdot (-)$ is a colimit in $\Set$.

  Conversely, assume that, for every $A\in \A_\fp$, the colimit cocone
  of the functor $\A(A,D-)$ is $\big(\A(A,c_i)\big)_{i\in \D}$. For
  every cocone $g_i\colon D_i\to G$ it is our task to prove that there
  exists a unique $g\colon C\to G$ with $g_i = g\cdot c_i$ for all
  $i$. We first prove uniqueness of $g$. If $g\cdot c_i=g'\cdot c_i$
  for all $i$, then $\A(A,g)\cdot \A(A,c_i)=\A(A,g')\cdot
  \A(A,c_i)$. Since the $\A(A,c_i)$ are jointly surjective, we obtain
  $\A(A,g)=\A(A,g')$. Since this holds for all $A\in \A_{\fp}$, and
  $\A_\fp$ is a generator, we have $g=g'$.

  Now $\big(\A(A,g_i)\big)_{i\in \D}$ forms a cocone of the functor
  $\A(A,-)\cdot D$. Consequently, there is a unique map
  $\phi_A\colon \A(A,C)\to \A(A,G)$ with $\phi_A\cdot \A(A,c_i)=\A(A,g_i)$
  for all $i\in \D$.
  
  For every morphism $h\colon A_1\to A_2$ between objects of $\A_\fp$
    the square on the right of the following diagram is commutative:
    \[
      \xymatrix{
        \A(A_1,D_i)\ar[rrr]^{\A(A_1,c_i)}
        &&&
        \A(A_1,C)\ar[rr]^{\phi_{A_1}}
        &&
        \A(A_1,G)
        \ar@{<-} `u[l] `[lllll]_-{\A(A_1,g_i)} [lllll]
        \\
        \A(A_2,D_i)\ar[u]^{\A(h,D_{i})}\ar[rrr]^{\A(A_2,c_i)}
        &&&
        \A(A_2,C)\ar[rr]^{\phi_{A_2}}\ar[u]_{\A(h,C)}
        &&
        \A(A_2,G)\ar[u]_{\A(h,G)}
        \ar@{<-} `d[l] `[lllll]^-{\A(A_2,g_i)} [lllll]}
    \]
    This follows from the commutativity of the left-hand square and
    the outside one combined with the fact that
    $\big(\A(A_2,c_i)\big)_{i\in \D}$, being a colimit cocone, is
    jointly epic.
  
  As a consequence, the morphisms
  \[
    A\xrightarrow{\phi_A(a)}C\qquad \text{with $a\colon A\to C$ in $\A_\fp/C $,}
  \]
  form a cocone for the canonical filtered diagram $D_C\colon \A_\fp/C \to \A$, of
  which $C$ is the colimit. Indeed, given a morphism $h$ in $\A_\fp/C$
  \[
    \xymatrix{
      A_1
      \ar[rd]_{a_1}
      \ar@{->}[rr]^-{h}
      &&
      A_2
      \ar[ld]^{a_2}
      \\
      &
      C
      &
    }
  \]
  we have
  $$
  \phi_{A_1}(a_1)=\phi_{A_1}(a_2\cdot h)=\phi_{A_1}\cdot\A(h,C)(a_2)=\A(h,G)\cdot
  \phi_{A_2}(a_2)=\phi_{A_2}(a_2)\cdot h.
  $$
  Thus there is a unique morphism $g\colon C\to G$ making for each
  $a\colon A\to C$ in $\A_\fp/C$ the following triangle commute:
  \[
    \xymatrix{
      & A
      \ar[dl]_{a}
      \ar[dr]^{\phi_A(a)}
      \\
      C
      \ar@{->}[rr]^{g}
      &&
      G
    }
  \]
  It satisfies $g\cdot c_i=g_i$ for all $i\in \D$. Indeed, fix $i$;
  for every $A\in \A_{\fp}$ and $b\colon A\to D_i$, we have
  $g_ib=\A(A,g_i)(b)=\phi_A\cdot
  \A(A,c_i)(b)=\phi_A(c_ib)=gc_ib$. And the morphisms
  $b\in \A_\fp/D_i$ are jointly epimorphic, thus $g_i=g\cdot c_i$.
  Thus $g$ is the desired factorization morphism.
\end{proof}

\begin{lemma}[Finitely generated objects collectively reflect filtered
  colimits of monomorphisms.]\label{L:refl2}
  Let $\A$ be an lfp category and $D\colon \D\to \A$ a filtered
  diagram of monomorphisms with ojects $D_i\, (i\in I)$. A cocone
  $c_i\colon D_i\to C$ of $D$ is a colimit iff for every $A\in \A_\fg$
  the cocone
  \[
    c_i\cdot (-)\colon \A(A,D_i) \longrightarrow  \A(A,C) \qquad (i\in I)
  \]
  is a colimit of the diagram $\A(A,D-)$ in $\Set$.
\end{lemma}
\begin{proof}
  If $(c_i)$ is a colimit, then since $\A(A,-)$ preserves filtered colimits of monomorphisms, the
  cocone $c_i\cdot(-)\colon \A(A,D_i)\to \A(A,C)$ is a colimit in $\Set$.

  Conversely, if for every $A\in \A_\fg$, the cocone $c_i\cdot (-)\colon \A(A,D_i)\to
  \A(A,C)$, $i\in I$, is a colimit of the
  diagram $\A(A,D-)$, then we have for every $A\in \A_\fp$ that the cocone $c_i\cdot
  (-), i\in I$, is a colimit of the diagram $\A(A,D-)$.
  Hence by Lemma~\ref{L:refl}, the cocone $(c_i)$ is a colimit.
\end{proof}
\begin{corollary} \label{C:refl}
  A functor $F\colon \A\to \B$ between lfp categories is finitary iff it
  preserves the canonical colimits: $FA = \colim FD_A$ for every object $A$ of $\A$.
\end{corollary}
\begin{proof}
  Indeed, in the notation of Lemma~\ref{L:refl} we are to verify that
  $Fc_i\colon FD_i\to FC$ ($i\in I$) is a colimit of $FD$. For this,
  taking into account that lemma and Remark \ref{R:refl}, we take any
  $B\in \B_\fp$ and prove that every morphism $b\colon B\to FC$
  factorizes essentially uniquely through $Fc_i$ for some $i\in
  \D$. Since $FC=\colim FD_C$ we have a factorization
  \[
    \vcenter{\xymatrix{
        & FA
        \ar[d]^{Fa}
        \\
        B
        \ar@{-->}[ur]^{b_0}
        \ar[r]_{b}
        & FC
      }}
    \qquad\quad (A\in \A_\fp)
  \]
  By Lemma~\ref{L:refl} there is some $i\in \D$ and $a_0\in \A(A,D_i)$ with
  $a=c_i\cdot a_0$ and hence $b=Fc_i\cdot (Fa_0\cdot b_0)$.
  The essential uniqueness is clear.
\end{proof}
\begin{notation} \label{N:image} Throughout the paper, given a morphism $f\colon
  X
  \to Y$ we denote by $\Im f$ the {\em image of $f$}, that is, any choice of the
  intermediate object defined by taking the (strong epi, mono)-factorization of
  $f$:
    \[
      f = (\xymatrix@1{X \ar@{->>}[r]^-e & \Im f \ar@{ >->}[r]^-m & Y}).
    \]
\end{notation}

We will make use of the next lemma in the proof of Proposition~\ref{P:finmono}. 

\begin{lemma}\label{L:im}
  In an lfp category, images of filtered colimits are directed unions
  of images.
\end{lemma}
More precisely, suppose we have a filtered diagram $D\colon  \D \to \A$ with objects $D_i\, (i\in I)$ and 
a colimit cocone $(c_i\colon  D_i \to C)_{i\in I}$. Given  a morphism
$f\colon  C \to B$, take the  factorizations of $f$ and all $f\cdot c_i$ as follows:
\begin{equation}\label{E:unionofimages}
  \vcenter{
    \xymatrix{
      D_i
      \ar[d]_{c_i} \ar@{->>}[rr]^-{e_i}
      &&
      \Im(f \cdot c_i) \ar@{ >->}[d]^{m_i} \ar@{-->}[ld]_-{d_i}
      \\
      C
      \ar@{->>}[r]^-e
      &
      \Im f \ar@{ >->}[r]^-{m}
      &
      B \ar@{<-} `d[l] `[ll]^-{f}
    }}\qquad \quad (i\in I)
\end{equation}
Then  the subobject
$m$ is the union of the subobjects $m_i$.
\begin{proof}
  We have the
  commutative diagram~\eqref{E:unionofimages}, where $d_i$ is the
  diagonal fill-in. Since $m\cdot d_i = m_i$, we see that $d_i$
  is  monic. Furthermore, for every connecting morphism $Dg\colon  D_i\to D_j$ we
  get a monomorphism  $\bar g\colon  \Im (f\cdot c_i)\monoto \Im (f\cdot c_j)$ as a diagonal fill-in
  in the diagram below:
  \[
    \xymatrix{
      D_i\ar@{->>}[r]^-{e_i}
      \ar[d]_{Dg}
      & \Im (f\cdot c_i)
      \ar@{>->}[dr]^-{d_i}
      \ar@{>-->}[d]_{\bar g}
      \\
      D_j
      \ar@{->>}[r]^-{e_j}
      &
      \Im (f\cdot c_j)
      \ar@{>->}[r]_-{d_j}
      &
      \Im f
    }
  \]
  Since $D$ is a filtered diagram, we see that the objects $\Im (f\cdot c_i)$ form a
  filtered diagram of monomorphisms; in fact, since $d_i$ and $d_j$ are monic
  there is at most one connecting morphism $\Im (f\cdot c_i) \to \Im (f\cdot c_j)$.

  In order to see that $m$ is the union of the subobjects $m_i$, let
  $d^{\prime}_i\colon \Im (f\cdot c_i)\monoto N$ and $n\colon N \monoto \Im f$ be
  monomorphisms such that $n\cdot d^{\prime}_i=d_i$ for every $i\in I$.
   \[
    \xymatrix{
      Di\ar@{->>}[r]^-{e_i}
      \ar[d]_{c_i}
      & \Im (f\cdot c_i)
      \ar@{>->}[d]^-{d^{\prime}_i}
      \ar@{>->}[r]^-{d_i}
      &
      \Im f
      \\
      C
      \ar@{-->}[r]^-{t}
      &
      N
      \ar@{>->}[ru]_-{n}
      &
    }
  \]
  Since $n$ is monic, the morphisms $d^{\prime}_i\cdot e_i$ clearly form a
  cocone of $D$, and this induces a unique morphism $t\colon C\to N$ such that $t\cdot
  c_i=d^{\prime}_i\cdot e_i$. Then $n\cdot t\cdot
  c_i=e\cdot c_i$; hence, $n\cdot t=e$. Since $n$ is monic, 
  it follows that it is an isomorphism, i.e.~the subobjects $\id_{\Im f}$ and
  $n$ are isomorphic. This shows that $m$ is the desired union.
\end{proof}

\section{Finitary and Finitely Bounded Functors}\label{sec:fin}

In this section we introduce the notion of a finitely bounded functor on a
locally presentable category, and investigate when these functors are precisely
the finitary ones.

\begin{definition}\label{D:bounded}
  A functor $F\colon  \A \to \B$ is called \emph{finitely bounded} provided
  that, given an object $A$ of $\A$, every finitely generated
  subobject of $FA$ in $\B$ factorizes through the $F$-image of a finitely generated
  subobject of $A$ in $\A$.
\end{definition}
  
\noindent
In more detail, given a monomorphism $m_0\colon  M_0 \monoto FA$ with $M_0
\in \B_{\mathsf{fg}}$ there exists a monomorphism $m\colon  M \monoto A$ with
$M \in \Afg$ and a factorization as follows:
  \[
  \xymatrix{
    & FM\ar[d]^{Fm}\\
    M_0\ar@{-->}[ru]\ar@{ >->}[r]_{m_0} & FA
  }
\]

\begin{example}\label{E:bounded}
  \begin{enumerate}
  \item\label{E:bounded:1} If $\B$ is the category of $S$-sorted sets,
    then $F$ is finitely bounded iff for every object $A$ of $\A$ and
    every element $x \in FA$ there exists a finitely generated
    subobject $m\colon  X \monoto A$  such that $x \in Fm[FX]$.
  \item\label{E:bounded:2} Let $\A$ be a category with (strong
    epi, mono)-factorizations. An object of $\A$ is finitely generated iff its
    hom-functor is finitely bounded. Indeed, by applying~\ref{E:bounded:1} we
    see that $\A(A,-)$ is finitely bounded iff for every morphism $f\colon A
    \to B$ there exists a factorization $f = m \cdot g$, where $m\colon A'
    \monoto B$ is monic and $A'$ is finitely generated. This implies that $A$ is
    finitely generated: for $f=\id_A$ we see that $m$ is invertible. Conversely,
    if $A$ is finitely generated, then we can take the (strong epi,
    mono)-factorization of $f$ and use that finitely generated objects are
    closed under strong quotients~\cite{AdamekR94}.
  \end{enumerate}
\end{example}

\begin{proposition}\label{P:finmono}
  Let $F$ be a functor between lfp categories preserving
  monomorphisms. Then $F$ is finitely bounded iff it preserves
  filtered colimits of monomorphisms.
\end{proposition}
\begin{proof}
  We are given lfp categories $\A$ and $\B$ and a functor $F\colon \A\to \B$
  preserving monomorphisms.

  \begin{enumerate}
  \item Let $F$ preserve filtered colimits of monomorphisms. Then, for
    every object $A$ we express it as a canonical filtered colimit of
    all $p\colon  P \to A$ in $\Afp/A$ (see Remark~\ref{R:prelim}\ref{I:canColim}). By
    Lemma~\ref{L:im} applied to $f = \id_A$ we see that $A$ is the
    colimit of its subobjects $\Im p$ where $p$ ranges over $\Afp/A$. 
    Hence, $F$ preserves this colimit,
    \[
      FA = \colim_{p \in \Afp/A} F(\Im p),
    \]
    and it is a colimit of monomorphisms since $F$ preserves monomorphisms. Given a finitely generated subobject $m_0\colon M_0\monoto FA$, we thus obtain
    some $p$ in $\Afp/A$ such that $m_0$ factorizes through the
    $F$-image of $\Im(p) \monoto A$. Hence $F$ is finitely bounded.
    
  \item Let $F$ be finitely bounded. Let $D\colon \D \to \A$ be a filtered
    diagram of monomorphisms with a colimit
    cocone:
    \[
      c_i\colon D_i\monoto C\qquad (i\in I).
    \]
    In order to prove that $Fc_i\colon FD_i\to FC$, $i\in I$, is a colimit cocone, we
    show that its image under $\B(B,-)$ is a colimit cocone for every finitely
    generated object $B$ in $\B$ (cf.~Lemma~\ref{L:refl2}). In other words,
    given $f\colon B\to FC$ with $B\in \B_\mathsf{fg}$ then
    \begin{enumerate}
    \item $f$ factorizes through $Fc_i$ for some $i$ in $I$, and
    \item the factorization is unique.
    \end{enumerate}
    
    We do not need to take care of (b): since every $c_i$ is monic by
    Remark~\ref{R:prelim}(4), so is every $Fc_i$. In order to prove (a),
    factorize $f\colon B\to FC$ as a strong epimorphism $q\colon B \epito M_0$
    followed by a monomorphism $m_0\colon M_0\monoto FC$. Then $M_0$ is
    finitely generated by Remark~\ref{R:prelim}(5). Thus, there exists a
    finitely generated subobject $m\colon M\monoto C$ with $m_0 = Fm \cdot u$ for
    some $u\colon  M_0 \to FM$. Furthermore, since $\A(M,-)$ preserves the
    colimit of $D$, $m$ factorizes as $m = c_i\cdot \overline{m}$ for some
    $i\in I$. Thus $F\overline{m}\cdot u\cdot q$ is the desired
    factorization:
    \[
      f
      = m_0 \cdot q
      = Fm \cdot u \cdot q
      = Fc_i \cdot F\overline{m} \cdot u \cdot q.
      \tag*{\endproofbox}
    \]
  \end{enumerate}
  \makeatletter
  \def\endproof{}
  \makeatother
\end{proof}
   
In the following theorem we work with an lfp category whose finitely generated
objects are finitely presentable. This holds e.g.~for the categories of sets,
many-sorted sets, posets, graphs, vector spaces, unary algebras on one operation
and nominal sets. Further examples are the categories of commutative monoids
(this is known as Redei's theorem~\cite{Redei65}, see Freyd~\cite{Freyd68} for a
rather short proof), positive convex algebras (i.e.~the Eilenberg-Moore algebras
for the (sub-)distribution monad on sets~\cite{SokolovaW15}), semimodules for
Noetherian semirings (see e.g.~\cite{bms13} for a proof). The category of
finitary endofunctors of sets also has this property as we verify in
Corollary~\ref{cor:setFunNotSStrict}.
  
On the other hand, the categories of groups, lattices or monoids do
not have that property.\smnote{Why lattices? Can we give a citation? TW: so we
  now know that the free lattice on 3 generators is infinite, but we could not
  find in the literature any statement about fp vs fg.}
A particularly simple counter-example is the slice category
$\Nat/\Set$; equivalently, this is the category of algebras with a set 
of constants indexed by $\Nat$. Hence, an object $a\colon \Nat\to A$ is
finitely generated iff $A$ has a finite set of generators,
i.e.~$A\setminus a[\Nat]$ is a finite set. It is finitely presentable
iff, moreover, $A$ is presented by finitely many relations, i.e.~the
kernel of~$a$ is a finite subset of $\Nat\times \Nat$.

\begin{theorem}\label{T:finbound}
  Let $\A$ be an lfp category in which every finitely generated object is
  finitely presentable ($\Afp=\Afg$). Then for all functors preserving
  monomorphisms from $\A$ to lfp categories we have the equivalence
  \[
    \text{finitary} \iff \text{finitely bounded}.
  \]
\end{theorem}
\begin{proof} Let $F\colon \A\to \B$ be a finitely bounded functor preserving
  monomorphisms, where $\B$ is lfp. We prove that $F$ is finitary. The
  converse follows from Proposition~\ref{P:finmono}.

  According to Corollary~\ref{C:refl} it suffices to prove that $F$ preserves
  the colimit of the canonical filtered diagram of every object $A$. The proof
  that $FD_A$ has the colimit cocone given by $Fp$ for all $p\colon P\to A$ in $\Afp /
  A$ uses the fact that this is a filtered diagram in the lfp category $\B$. By
  Remark~\ref{R:refl}, it is therefore sufficient to prove that for every object
  $C\in \B_\mathsf{fp}$ and every morphism $c\colon C\to FA$ we have the following two
  properties:
  \begin{enumerate}
  \item\label{T:finbound:1}
    $c$ factorizes through some of the colimit maps 
    \[
      \vcenter{
        \xymatrix{
          & FP\ar[d]^{Fp}\\
          C\ar@{-->}[ur]^{u}\ar[r]_-{c}
          & FA
        }
      }
      \qquad
      (P\in\Afp),
    \]
  \item\label{T:finbound:2} given another such factorization, $c=Fp\cdot v$, then $u$ and
    $v$ are merged by some connecting morphism; i.e.~we have a
    commutative triangle
    \[
      \vcenter{
        \xymatrix{
          P\ar[rr]^-{h}\ar[rd]_{p}
          &&
          P'\ar[ld]^{p'}\\
          &A&
        }}
      \qquad
      (P,P'\in\Afp)
    \]
    with $Fh\cdot u = Fh\cdot v$.
  \end{enumerate}

  Indeed, by applying  Lemma~\ref{L:im} to
  $f = \id_A$, we see that the monomorphisms $m_p\colon \Im p \monoto A$ for $p \in
  \A_\fp/A$ form a colimit cocone of a diagram of monomorphisms. By
  Proposition~\ref{P:finmono}, $F$ preserves this colimit, 
  therefore any  $c\colon C\to FA$ factorizes through some
  $Fm_p\colon F(\Im p)\to FA$. Observe that, since $\Afg=\Afp$, we know by
  Remark~\ref{R:prelim}(5) that every $\Im p$ is finitely presentable,
  hence the morphisms $m_p$ are colimit injections and all $e_p\colon P\epito \Im p$ are connecting
  morphisms of $D_A$. Consequently, (1) is clearly satisfied.
  Moreover, given $u,v\colon C\to FP$ with $Fp \cdot u= Fp \cdot v$, we have
  that $Fe_p\cdot u=Fe_p\cdot v$, since $Fm_p$ is monic, thus (2) is
  satisfied, too.
\end{proof}

\begin{remark}
  Conversely, if every functor from $\A$ to an lfp category fulfils
  the equivalence in the above theorem, then $\Afp=\Afg$. Indeed, for
  every finitely generated object $A$, since $F=\A(A,-)$ preserves
  monomorphisms, we can apply Proposition~\ref{P:finmono} and conclude
  that $F$ is finitary, i.e.~$A\in \Afp$.
\end{remark}

\begin{example}\label{E:unary}
  For $\Un$, the category of algebras with one unary operation, we
  present a finitely bounded endofunctor that is not finitary. Since
  in $\Un$ finitely generated algebras are finitely presentable, this
  shows that the condition of preservation of monomorphisms cannot be
  removed from Theorem~\ref{T:finbound}.

  Let $C_p$ denote the algebra on $p$ elements whose operation forms a
  cycle. Define $F\colon\Un\to \Un$ on objects by
  \[
    FX =
      \begin{cases}
        C_1 + X &
        \text{if $\Un(C_p,X)=\emptyset$ for some prime $p$,}
        \\
        C_1 & \text{else.}
      \end{cases}
  \]
  Given a homomorphism $f\colon  X \to Y$ with $FY=C_1 + Y$, then also
  $FX=C_1 + X$; indeed, in case $FX = C_1$ we would have
  $\Un(C_p, X) \neq \emptyset$ for all prime numbers $p$, and then the
  same would hold for $Y$, a contradiction. Thus we can put
  $Ff=\id_{C_1} + f$. Otherwise $Ff$ is the unique homomorphism to
  $C_1$.
  
  \begin{enumerate}
  \item We now prove that  $F$ is finitely bounded. Suppose we
  are given a finitely generated subalgebra $m_0\colon M_0\monoto FX$. If
  $FX = C_1$ then take $M = \emptyset$ and $m\colon  \emptyset \monoto X$
  the unique homomorphism. Otherwise we have $FX = C_1 + X$, and we
  take the preimages of the coproduct injections under $Ff$ to see that $m_0 = u
  + m$, where $u$ is the unique homomorphism into the terminal algebra
  $C_1$ as shown below:
  \[
    \xymatrix@C+2pc{
      M' \ar@{ >->}[d] \ar[r]^-u & C_1 \ar@{ >->}[d] \\
      M_0 \ar[r]^-{m_0} & C_1 + X \\
      M \ar@{ >->}[u] \ar[r]_-m & X \ar@{ >->}[u]
    }
  \]
  Then we obtain the desired factorization of $m_0$:
  \[
    \xymatrix{
      & C_1 + M = FM\ar[d]^{\id_{C_1} + m = Fm}
      \\
      M_0 = M' + M
      \ar[r]_-{u + m}
      \ar[ru]^-{u + M}
      &
      C_1 + X = FX
    }
  \]
  
  \item However, $F$ is not finitary; indeed, it does not preserve the
  colimit of the following chain of inclusions
  \[
    C_2\hookrightarrow C_2+C_3\hookrightarrow
    C_2+C_3+C_5\hookrightarrow \cdots
  \]
  since every object $A$ in this chain is mapped by $F$ to $C_1+A$ while its
  colimit $X = \coprod_{i\text{ prime}} C_i$ is mapped to $C_1$.
  \end{enumerate}
\end{example}

\vspace{2mm}

We now turn to the question for which lfp categories $\A$ the equivalence
\[
  \text{finitary} \iff \text{finitely bounded}
\]
holds for \emph{all} functors with domain $\A$.

In the following we call a morphism $u\colon X\to Y$ \emph{finitary}
if it factorizes through a finitely presentable object:
\begin{equation}\label{eq:finitary}
  \vcenter{
  \xymatrix{
    & C \mathrlap{~\in \A_\fp}
    \ar[d]^{w}
    \\
    X\ar[ru]^v
     \ar[r]_{u}
    & Y
  }}
\end{equation}
\begin{example}\label{E:loop}
  In the category of graphs consider the following graph on $\Nat$:
  \[
    \xymatrix{
      0
      \ar@(ul,dl)[]
      & 1
      \ar[r]
      & 2
      \ar[r]
      & 3
      \ar[r]
      & \cdots
    }
  \]
  The constant self-map of value $0$ is finitary, but no other endomorphism on this
  graph is finitary.
\end{example}

\begin{remark}
  \begin{enumerate}
  \item If a morphism in an lfp category has a finitely presentable image
    (see Notation~\ref{N:image}), then it is of course finitary.

  \item The converse, namely that every finitary morphism has a
    finitely presentable image, holds whenever $\A_\fp$ is closed under subobjects and
    $\A_\fp= \A_\fg$. Indeed, given a finitary morphism $u\colon X\to Y$, let
    $w\cdot v$ be a factorization through a finitely presentable object $C$.
    Take a (strong epi, mono)-factorization $v = v_2\cdot v_1$ of
    $v$:
    \[
      \xymatrix{
        & C_1
        \ar@{>->}[r]^{v_2}
        \ar[d]^{d}
        & C
        \ar[d]^{w}
        \\
        X
        \ar@{->>}[r]
        \ar@{->>}[ur]^{v_1}
        & \Im(u)
        \ar@{>->}[r]
        & Y
      }
    \]
    Then $C_1$ is finitely presentable and the diagonal fill-in $d$ is strongly
    epic thus, $\Im(u)$ is finitely presentable. This holds e.g.~for sets,
    graphs, posets, vector spaces and semilattices.
  \end{enumerate}
\end{remark}
\begin{definition} \label{D:strict}
  An lfp category is called
  \begin{enumerate}
  \item \emph{semi-strictly lfp} if every object has a finitary endomorphism;
  \item \emph{strictly lfp} if every object has, for each finitely generated
    subobject $m$, a finitary endomorphism $u$ fixing that subobject
    (i.e.~$u\cdot m = m$).
  \end{enumerate}
\end{definition}
\begin{remark} \label{R:strict}
  \begin{enumerate}
  \item \label{R:strict:2semi} `strictly' implies `semi-strictly' due to $0\in
    \A_\fp$: use the image of the unique $b\colon 0\to A$.
  \item \label{R:strict:str} An lfp category is strictly lfp iff for every morphism
    $b\colon B\to A$ with $B\in \A_\fp$ there exist morphisms $b'\colon B'\to A$
    and $f\colon A\to B'$ with $B'\in \A_\fp$ such that the square below
    commutes.
    \[
      \xymatrix{
        B
        \ar[r]^b
        \ar[d]_b
        & A
        \\
        A
        \ar[r]_f
        & B'
        \ar[u]_{b'}
      }
    \]
    Indeed, this condition is necessary: choose, for the image $m$ of $b$,
    a finitary $u\colon A\to A$ with $m=u\cdot m$, thus $b=u\cdot b$.
    We have a
    factorization $u=b'\cdot f$ where $b'\colon B'\to A$ has a finitely
    presentable domain.

    The condition is also sufficient: given a square as above, the morphism
    $u=b'\cdot f$ is finitary and $b = u\cdot b$.

  \item \label{R:strict:semi} An lfp category is semi-strictly lfp iff for every
    morphism $b\colon B\to A$ with $B\in \A_\fp$ there exists a factorization of
    $b$ through a morphism $b'\colon B'\to A$ with $B'\in \A_\fp$ such that
    $\A(B',A) \neq \emptyset$.
    \[
      \xymatrix{B\ar[rd]_b\ar@{-->}[rr]&&B'\ar@<2pt>[ld]^{b'}\\
        &A\ar@<2pt>[ur]^f&}
    \]
    Indeed, this condition is necessary: given a finitary morphism
    $u\colon A\to A$ we have $u=w\cdot v$ as
    in~\eqref{eq:finitary}. Moreover, $B'=B+C$ is finitely presentable
    since both $B$ and $C$ are. Put $b'=[b,w]\colon B'\to A$ and
    \[
      f = \big(
      A\xrightarrow{~v~}
      C\xrightarrow{~\inr~}
      B+C
      \big),
    \]
    where $\inr$ is the right-hand coproduct injection.
    Then $b$ factorizes through $b'$ via the left-hand coproduct
    injection $\inl\colon B\to B+C$.

    The condition is also sufficient: consider $b\colon 0 \to A$ and put $a=b'\cdot f$.

  \item \label{R:strict:fpfg} In every strictly lfp category we have
    $\Afg=\Afp$. Indeed, given $A\in \Afg$ express it as a strong quotient
    $b\colon B\epito A$ of some $B\in \Afp$, see Remark \ref{R:prelim}(5). Then the
    equality $b=b'\cdot f\cdot b$ in \ref{R:strict:str} above implies $b'\cdot f=\id$. Thus,
    $A$ is a split quotient of a finitely presentable object $B'$, hence, $A$ is
    finitely presentable by Remark \ref{R:prelim}(6).
  \end{enumerate}
\end{remark}
\begin{examples}\label{E:set}
 \begin{enumerate}
  \item $\Set$ is strictly lfp: given $b\colon B\to A$ with $B\not=\emptyset$
    factorize it as $e\colon B\epito \Im b$ followed by a split monomorphism
    $b'\colon \Im
    b\to A$. Given a splitting, $f\cdot b'=\id$, we have $b=b'\cdot f\cdot b$.
    The case $B=\emptyset$ is trivial: for $A\not=\emptyset$, $b'$ may be any
    map from a singleton set to $A$.
    
  \item Vector spaces (over a given field) form a strictly lfp
    category. This can be seen directly quite easily, we show this in
    Example \ref{E:strict}\ref{e:Vec} as a consequence of
    Proposition~\ref{P:semi-simple}.
    
  \item For every finite group $G$ the category $G$-$\Set$ of sets
    with an action of $G$ is strictly lfp. This category is equivalent
    to that of presheaves on $G^{\text{op}}$, see
    Lemma~\ref{L:gpd}. 

  \item \label{R:strict:zeroobject} Every lfp category with a zero
    object $0\cong 1$ is semi-strictly lfp. This follows from the fact
    that $0$ is finitely presentable and every object $A$ has the
    finitary endomorphism
    $\big( A \to 1 \cong 0 \to A \big)$. Examples
    include the categories of monoids and groups, which are not
    strictly lfp because in both cases the classes of finitely
    presentable and finitely generated objects differ. 
    
    A bit more generally: let an lfp category $\A$ have a finitely
    presentable terminal object from which morphisms exist to all
    objects outside of $\A_{\fp}$. Then it is semi-strictly lfp. For
    example, the category of posets is semi-strictly lfp.

  \item \label{E:strictNotSemi} An example of an lfp category $\A$ which fulfils
    $\A_\fp = \A_\fg$ but is not semi-strictly lfp is the category of graphs.
    The subgraph of the graph of Example~\ref{E:loop} on $\Nat\setminus\{0\}$ has no
    finitary endomorphism. Another such example is the category of nominal sets
    which is discussed in Example~\ref{ex:Nom}.
  \end{enumerate}
\end{examples}
We will see other examples (and non-examples) below. The following
figure shows the relationships between the different properties:
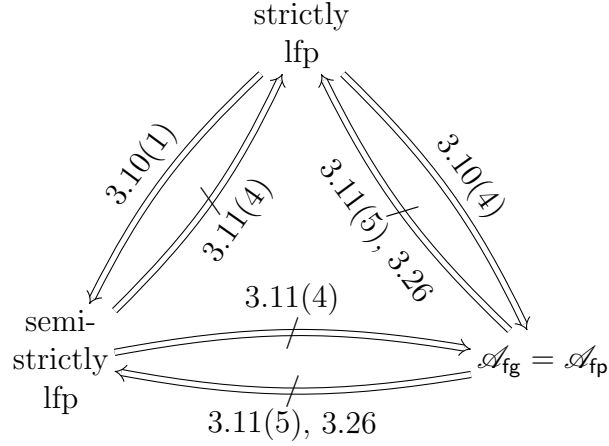
\[
  \begin{tikzpicture}[
    every node/.append style={
      align=center,
      anchor=center,
      inner sep=2pt,
      minimum height=2em,
    }
    ]
    \node (strict) at (90:25mm) {strictly\\ lfp};
    \node (semi) at (-150:37mm) {semi-\\ strictly\\ lfp};
    \node (fpfg)  at (-30:37mm){$\Afg = \Afp$};
    \begin{scope}[
      every edge/.append style={-implies,double equal sign distance,draw,bend left=10},
      every node/.style={sloped,above},
      ]
    \path (strict) edge[bend right] node {\ref{R:strict}\ref{R:strict:2semi}} (semi);
    \path (semi) edge[bend right] node[anchor=center] {/}
      node[below,outer sep=1mm] {\ref{E:set}\ref{R:strict:zeroobject}}
      (strict);
    \path (strict) edge node {\ref{R:strict}\ref{R:strict:fpfg}} (fpfg);
    \path (fpfg) edge node[anchor=center] {/}
      node[below,outer sep=1mm] {\ref{E:set}\ref{E:strictNotSemi}, \ref{ex:Nom}}
    (strict);
    \path (fpfg) edge
      node[anchor=center] {/}
      node[below,outer sep=1mm] {\ref{E:set}\ref{E:strictNotSemi}, \ref{ex:Nom}}
      (semi);
    \path (semi) edge
    node[anchor=center] {/}
    node[outer sep=1mm] {\ref{E:set}\ref{R:strict:zeroobject}}
    (fpfg);
    \end{scope}
\end{tikzpicture}
\]
\begin{theorem}\label{T:boundstrict}
  Let $\A$ be a strictly lfp category, and  $\B$  an lfp category
  with $\B_{\mathsf{fg}} = \B_{\mathsf{fp}}$. Then for all functors from
  $\A$ to $\B$ we have the equivalence
  \[
    \text{finitary} \iff \text{finitely bounded}.
  \]
\end{theorem}
\begin{proof} 
  ($\Longrightarrow$) Let $F\colon  \A \to \B$ be finitary. By
  Remark~\ref{R:strict}\ref{R:strict:fpfg} we know that $\Afp = \Afg$. Given a
  finitely generated subobject $m\colon  M \monoto FA$, write $A$ as the
  directed colimit of all of its finitely generated subobjects
  $m_i\colon  A_i \monoto A$. Since $F$ is finitary, it preserves this
  colimit, and since $M$ is finitely generated, whence finitely
  presentable, we obtain some $i$ and some $f\colon  M \to FA_i$ such that
  $Fm_i \cdot f = m$ as desired.

  ($\Longleftarrow$) Suppose that $F\colon  \A \to \B$ is finitely bounded. We verify the
  two properties~\ref{T:finbound:1} and~\ref{T:finbound:2} in the proof of
  Theorem~\ref{T:finbound}. In order to verify~\ref{T:finbound:1}, let $c\colon  C \to
  FA$ be a morphism with $C$ finitely presentable. Then we have the finitely
  generated subobject $\Im c \monoto FA$, and this factorizes through $Fm\colon  FM
  \to FA$ for some finitely generated subobject $m\colon  M \monoto A$ since $F$ is
  finitely bounded. Then $c$ factorizes through $Fm$, too, and we are done since
  $M$ is finitely presentable by Remark~\ref{R:strict}(4).

  To verify~(2), suppose that we have $u,v\colon  C \to FB$ and
  $b\colon  B \to A$ in $\Afp/A$ such that $Fb \cdot u = Fb \cdot v$. Now
  choose $f\colon  A \to B'$ with $b = b' \cdot (b \cdot f)$ (see
  Remark~\ref{R:strict}\ref{R:strict:str}). Put $h = f\cdot b$ to get $b = b'\cdot h$
  as required. Since $Fb\cdot u = Fb\cdot v$, we conclude
  $Fh\cdot u = Ff \cdot Fb \cdot u = Ff \cdot Fb \cdot u = Fh\cdot v$.
\end{proof}
\begin{corollary}
  A functor between strictly lfp categories is finitary iff it is
  finitely bounded.
\end{corollary}

\begin{remark}
  Consequently, a set functor $F$ is finitary if and only if it is
  finitely bounded. The latter means precisely that every element of
  $FX$ is contained in $Fm[FM]$ for some finite subset $m\colon  M \subto
  X$.

  This result was formulated already in~\cite{AdamekP04}, but the
  proof there is unfortunately incorrect.
\end{remark}

\begin{openproblem}
  Is the above implication an equivalence? That is, given an lfp category $\A$
  such that every finitely bounded functor into lfp categories is finitary, does
  this imply that $\A$ is strictly lfp?
\end{openproblem}
\begin{theorem}
  \label{T:equiv2semstrict}
  Let $\A$ be an lfp category such that for functors
  $F\colon  \A \to \Set$ we have the equivalence
   \[
    \text{finitary} \iff \text{finitely bounded}.
  \]
  Then $\A$ is semi-strictly lfp and $\Afg = \Afp$.
\end{theorem}
\proof
  The second statement easily follows from Example~\ref{E:bounded}\ref{E:bounded:2}.
  Suppose that $\A$ is an lfp category such that the above equivalence
  holds for all functors from $\A$ to $\Set$. Then the same
  equivalence holds for all functors $F\colon  \A \to \Set^S$, for $S$ a set of
  sorts. To see this, denote by $C\colon  \Set^S \to \Set$ the functor forming
  the coproduct of all sorts. It is easy to see that $C$ creates
  filtered colimits. Thus, a functor $F\colon  \A \to \Set^S$ is finitary iff
  $C\cdot F\colon  \A \to \Set$ is. Moreover, $F$ is finitely bounded iff
  $C\cdot F$ is; indeed, this follows immediately from Example~\ref{E:bounded}\ref{E:bounded:1}.
  
  We proceed to prove the semi-strictness of $\A$. Put
  $S= \Afp$. Given a
  morphism
  \[
    b\colon B\to A\qquad\text{with $B\in \Afp$}
  \]
  we present $b'$ and $f$ as required in Remark~\ref{R:strict}\ref{R:strict:str}. Define a functor
  $F\colon \A\to Set^S$ on objects $Z$ of $\A$ by
  \[
    FZ=\begin{cases}
      \mathds{1}+(\A(s,Z))_{s\in S}& \text{if $\A(A,Z)=\emptyset$}\\
      \mathds{1}&\text{else,}
    \end{cases}
  \]
  where $ \mathds{1}$ denotes the terminal $S$-sorted set. Given a
  morphism $f\colon  Z \to Z'$ we need to specify $Ff$ in the case where
  $\A(A,Z')= \emptyset$: this implies $\A(A,Z)  =\emptyset$ and we put
  \[
    Ff = \id_\mathds{1} + (\A(s,f))_{s \in S}.
  \]
  Here $\A(s,f)\colon  \A(s, Z) \to \A(s, Z')$ is given by $u \mapsto f
  \cdot u$, as usual. It is easy to verify that $F$ is a well-defined
  functor.

  \begin{enumerate}
  \item Let us prove that $F$ is finitely bounded. The category
  $\Set^S$ is lfp with finitely generated objects $(X)_{s\in S}$
  precisely those for which the set $\coprod_{s \in S} X_s$ is
  finite. Let $m_0\colon  M_0 \monoto FZ$ be a finitely generated
  subobject. We present a finitely generated subobject $m\colon  M \monoto
  Z$ such that $m_0$ factorizes through $Fm$. This is trivial in the
  case where $\A(A,Z) \neq \emptyset$: choose any finitely generated
  subobject $m\colon  M \monoto Z$ (e.g.~the image of the unique morphism
  from the initial object to $Z$: cf.~Remark~\ref{R:prelim}(5)). Then
  $Fm$ is either $\id_\mathds{1}$ or a split epimorphism, since $FZ =
  \mathds{1}$ and in $FM$ each sort is non-empty. Thus, we have $t$ with $Fm\cdot t=\id$ and $m_0$
  factorizes through $Fm$:
  \[
    \xymatrix{
      & FM \ar@<-2pt>@{->>}[d]_{Fm} \\
      M_0\ar[ru]^-{t \cdot m_0} \ar@{ >->}[r]_-{m_0}
      &
      FZ = \mathds{1}
      \ar@<-2pt>@{>->}[u]_t
      }
  \]
  In the case where $\A(A,Z) = \emptyset$ we have $m_0 = m_1 + m_2$
  for subobjects
  \[
    m_1 \colon  M_1 \monoto \mathds{1}
    \qquad
    \text{and}
    \qquad
    m_2\colon  M_2 \monoto (\A(s,Z))_{s \in S}.
  \]
  For notational convenience, assume $(M_2)_s \subseteq \A(s, Z)$ and $(m_2)_s$
  is the inclusion map for every $s \in S$. Since $M_0$ is finitely generated,
  $M_2$ contains only finitely many elements $u_i\colon  s_i \to Z$, $i = 1, \ldots,
  n$. Factorize $[u_1, \ldots, u_n]$ as a strong epimorphism $e$ followed by a
  monomorphism $m$ in $\A$ (see Remark~\ref{R:prelim}\ref{I:factorization}):
  \[
    \xymatrix{
      \coprod_{i=1}^n s_i \ar@{->>}[r]^-{e}
      &
      M \ar@{ >->}[r]^-m
      &
      Z
    }.
  \]
  Then $\A(A,M) = \emptyset$, therefore $Fm = \id_\mathds{1} +
  (\A(s,m))_{s\in S}$. Since every element $u_i\colon  s_i \to Z$ of $M_2$
  factorizes through $m$ in $\A$, we have
  \[
    u_i = m \cdot u_i'
    \quad
    \text{for $u_i'\colon  s_i \to M$ with $[u_1', \ldots, u_n'] = e$.}
  \]
Let $v:M_2 \to \A(s,M)$ be the $S$-sorted map taking each $u_i$ to $u'_i$.
  Then the inclusion map $m_2\colon  M_2 \to (\A(s,Z))_{s \in S}$ has the
  following form
  \[
    m_2 = \left(M_2 \xrightarrow{v} (\A(s,M))_{s \in S}
      \xrightarrow{(\A(s,m))_{s\in S}} (\A(s,Z))_{s\in S}\right).
  \]
  The desired factorization of $m_0 = m_1 + m_2$ through $Fm =
  \id_\mathds{1} + (\A(s,m))_{s\in S}$ is as follows:
  \[
    \xymatrix@C+2pc{
      &
      \mathds{1} + (\A(s,M))_{s\in S}
      \ar[d]^{\id + (\A(s,m))_{s\in S}}
      \\
      M_0 = M_1 + M_2
      \ar@{ >->}[r]_-{m_0 = m_1 + m_2}
      \ar[ru]^-{m_1 + v}
      &
      \mathds{1} + (\A(s,Z))_{s \in S}
      }
  \]

\item We thus know that $F$ is finitary, and we will use this to prove that $\A$
  is semi-strictly lfp. That is, as in Remark~\ref{R:strict}\ref{R:strict:semi}
  we find $b'\colon  B' \to A$ in $\Afp/A$ through which $b$ factorizes and which
  fulfils $\A(A,B') \neq \emptyset$. Recall from Remark~\ref{R:prelim}(2) that
  $A = \colim D_A$. Our morphism $b$ is an object of the diagram scheme $\Afp/A$
  of $D_A$. Let $D_A'$ be the full subdiagram of $D_A$ on all objects $b'$ such
  that $b$ factorizes through $b'$ in $\A$ (that is, such that a connecting
  morphism $b \to b'$ exists in $\Afp/A$). Then $D_A'$ is also a filtered
  diagram and has the same colimit, i.e.~$A = \colim D_A'$. Since $F$ preserves
  this colimit and $FA = \mathds{1}$, we get
  \[
    \mathds{1} \cong \colim FD_A'.
  \]
 
  Assuming that $\A(A,B') = \emptyset$ for all $b'\colon  B'\to A$ in
  $D_A'$, we obtain a contradiction: the objects of $FD_A'$ are
  $\mathds{1} + (\A(s,B'))_{s\in S}$, and since for every $s \in S$
  the functor $\A(s, -)$ is finitary, the colimit of all $\A(s,B')$ is
  $\A(s,A)$. We thus obtain an isomorphism
  \[
    \mathds{1} \cong \mathds{1} + (\A(s,A))_{s\in S}.
  \]
  This means $\A(s,A) = \emptyset$ for all $s \in S$, in particular
  $\A(B,A) = \emptyset$, in contradiction to the existence of the given
  morphism $b\colon  B \to A$.

  Therefore, there exists $b'\colon  B' \to A$ in $D_A'$, i.e.~$b'$ through
  which $b$ factorizes with $\A(A,B') \neq \emptyset$, as required.
  \endproof
\end{enumerate}

We now present examples of strictly lfp categories. All of them happen
to be either atomic toposes or semi-simple (aka atomic) abelian
categories. Recall that an object $A$ is called \textit{simple}, or an
\emph{atom}, if it has no nontrivial subobject. That is, every
subobject of $A$ is either invertible or has the initial object as a
domain.

\begin{definition}\label{D:semi-simple}
  A category is called {\em semi-simple} or {\em atomic} if every
  object is a coproduct of simple objects.
\end{definition}
\begin{proposition}\label{P:semi-simple}
  Let a semi-simple, cocomplete category have only finitely many
  simple objects (up to isomorphism), all of them finitely
  presentable. Then it is strictly lfp.
\end{proposition}

\begin{proof}
  \begin{enumerate}
  \item The given category $\A$ is lfp. Indeed, it is cocomplete and every finite coproduct of simple objects is finitely presentable. Moreover, every object $\coprod_{i\in  I}A_i$, $A_i$ simple, is a filtered colimit of finite subcoproducts. Conversely, every finitely presentable object is, obviously, a split quotient of a finite coproduct of simple objects. Thus, for the countable set $M$ representing all these finite coproducts we see that $\Afp$ consists of split quotients of objects in $M$. Therefore $\Afp$ is essentially a
set: split quotients of any object $X$ correspond bijectively to idempotent endomorphisms of $X$, and thus form a set. Hence, $\A$ is lfp.

\item\label{P:semi-simple:2} Let $b\colon B\to A=\coprod_{i\in I}A_i$ be a morphism with all
  $A_i$ simple and $B$ finitely presentable. Then $b$ factorizes
  through a finite subcoproduct $a_J\colon \coprod_{i\in J}A_i\to A$
  ($J\subseteq I$ finite), say, $b=a_J\cdot b'$. Since $\A$ has
  essentially only a finite set of simple objects, $J$ can be chosen
  so that each $A_i$ is isomorphic to some $A_j$, $j\in
  J$. Consequently, there exists a morphism
  $g\colon \coprod_{i\in I\setminus J}A_i \to \coprod_{j\in
    J}A_j$. The following composite $u\colon A\to A$
  $$
  A = \big(\coprod_{j\in  J}A_j+\coprod_{i\in  I\setminus J}A_i\big)
  \xrightarrow{[\id,g]}
  \coprod_{j\in  J}A_j\xrightarrow{a_J} A
  $$
  is finitary and fulfils, since $[\id,g]\cdot a_J=\id$, the desired equation
  $$u\cdot b=a_J\cdot [\id,g]\cdot a_J\cdot b'=a_J\cdot b'=b.$$
\end{enumerate}
\end{proof}

\begin{examples}\label{E:strict}
  \begin{enumerate}
  \item\label{e:SetS}$\Set^S$ is strictly lfp iff $S$ is
    finite.  Indeed, the sufficiency is a clear consequence of Proposition \ref{P:semi-simple}.
    Conversely, if $S$ is
    infinite then the identity on the terminal object, which is its
    unique endomorphism, is not finitary, whence $\Set^S$ is not
    semi-strictly lfp.
    
  \item\label{e:Vec} For every field $K$ the category
    $K$-$\mathsf{Vec}$ of vector spaces is strictly lfp. Indeed, the
    simple spaces are those of dimension $0$ or $1$, and every space is a
    coproduct of copies of $K$.

  \item\label{e:Mod} We recall that a ring $R$ is called \emph{semi-simple}
    if the category $R$-$\mathsf{Mod}$ of left modules is
    semi-simple. For example, the matrix ring $K^{(n)}$ for every field
    $K$ and every finite $n$ is semi-simple.
   
    The category $R$-$\mathsf{Mod}$ is strictly lfp for every finite
    semi-simple ring $R$. Indeed, every simple module $A$ is a
    quotient of the module $R$: in case $A\not=\mathbf{0}$, choose
    $a\in A\setminus \{0\}$. Since $Ra$ is a submodule of $A$, we conclude
    $$
    A=Ra \cong R/\mathord{\sim}
    $$
    where $\sim$ is the congruence defined by $x\sim y$ iff $Rx=Ry$.
   
    Each quotient module $R/\mathord{\sim}$ is finitely presentable. Indeed, let
    $a_i\colon A_i\to A$, $i\in I$, be a filtered colimit and
    $f\colon R/\mathord{\sim} \to A$ a homomorphism. Since $R/\mathord{\sim}$ is
    finite, $f$ factorizes in $\Set$ through $a_j$ for some $j\in J$:
    $f=a_j\cdot f'$. It remains to choose $j$ so that
    $f'\colon R/\mathord{\sim} \to A_j$ is a homomorphism. Given $r, s\in R$ we
    know that $rf([s])=f([rs])$, thus $a_j$ merges $rf'([s])$ and
    $f'([rs])$. Our colimit is filtered, hence for the given pair we
    can assume, without loss of generality, that
    $rf'([s])=f'([rs])$. Moreover, since $R\times R$ is finite, this
    assumption can be made for all pairs $(r,s)$ at once. That is, by
    a suitable choice of $j$ we achieve that $f'$ preserves scalar
    multiplication. A completely analogous argument shows that $j$ can
    be chosen so that, moreover, $f'$ preserves addition. Thus, it is
    a homomorphism.
   
  \item\label{e:topos} A Grothendieck topos is called \textit{atomic},
    see \cite{BarrDiaconescu}, if it is semi-simple. For example, the
    presheaf topos $\Set^{{\cal{C}}^\text{op}}$ is atomic iff
    $\cal{C}$ is a groupoid, i.e.~its morphisms are all invertible,
    see Sect.~7(2) in op.~cit. It follows from the Proposition~\ref{P:semi-simple}
    that every atomic Grothendieck topos with a finite set of finitely
    presentable atoms (up to isomorphism) is strictly lfp.
   
    More atomic toposes can be found in
    \cite[Example~3.5.9]{JohnstoneElephant2vol}. Not all atomic Grothendieck
    toposes are semi-strictly lfp. See Example~\ref{ex:Nom} below: in the
    category of nominal sets (aka the Schanuel topos), the set of atoms
    is infinite. Next we provide a class of examples of strict lfp toposes, see
    also Example~\ref{E:gpd-lambda} below.
  \end{enumerate}
\end{examples}
\begin{lemma}\label{L:gpd}
The category of presheaves on a finite groupoid is strictly lfp.
\end{lemma}
\begin{proof} In view of Example \ref{E:strict}(4) all we need proving is that
  for every finite groupoid $\G$
  the category $\Set^{\G^{\text{op}}}$ has, up to isomorphism, a finite set of finitely presentable
  atoms.

(1)~Put $S=\obj \G$. Then the category
  $\Set^{\G^{\text{op}}}$ can be considered as a
  variety of $S$-sorted unary algebras. The signature is given by
  the set of all morphisms of $\G^{\text{op}}$: every morphism $f\colon  X
  \to Y$ of $\G^{\text{op}}$ corresponds to an operation symbol of
  arity $X \to Y$ (i.e.~variables are of sort $X$ and results of sort
  $Y$). This variety is presented by the equations corresponding to the composition in
  $\G^{\text{op}}$: represent $g\cdot f=h\colon X\to Y$ in
  $\G^{\text{op}}$ by $g(f(x))=h(x)$ for a variable $x$ of sort
  $X$. Moreover, for every
  object $X$, add the equation $\id_X(x) = x$ with $x$ of sort $X$. 
  
  For every algebra $A$ and every element $x\in A$ of sort $X$ the
  subalgebra which $x$ generates is denoted by $A^x$. Denote by $\sim_A$ the
  equivalence on the set of all elements of $A$ defined by $x\sim_Ay$
  iff $A^x=A^y$. If $I(A)$ is a choice class of this equivalence, then
  we obtain a representation of $A$ as the following coproduct:
  \[
    A=\coprod_{x\in I(A)}A^x.
  \]
  This follows from $\G$ being a groupoid: whenever
  $A^x\cap A^y\not=\emptyset$, then $x\sim_A y$.
  
  Moreover, for every homomorphism $h\colon A\to B$ there exists a function
  $h_0\colon I(A) \to I(B)$ such that on each $A^x$, $x\in I(A)$, $h$
  restricts to a homomorphism $h_0\colon A^x\to B^{h(x)}$. Indeed,
  define $h_0(x)$ as the representative of $\sim_B$ with
  $B^{h(x)}=B^{h_0(x)}$.

  \medskip\noindent (2)~Given $x\in A$ of sort $X$, the algebra $A^x$ is a
  quotient of the representable algebra $\G(-,X)$. Indeed, the Yoneda
  transformation corresponding to $x$, an element of $A^x_X$ of sort $X$, has
  surjective components (by the definition of $A^x$).
  
  Observe that every representable algebra has only finitely many
  quotients. This follows from the fact that $\G(-,X)$ has finitely
  many elements, hence, finitely many equivalence relations exist
  on the set of all elements.

  \medskip\noindent (3)~We conclude that the finite set $\B$
  of all algebras representing quotients of representable algebras
  $\G(-,X)$ consists of finitely presentable algebras. Moreover, every
  algebra is a coproduct of algebras from $\B$. 
%
\end{proof}
%
%

\begin{remark}
  Recall from \cite[Proposition 2.30]{AdamekR94} that \emph{pure subobjects}
  $b\colon B\monoto A$ in an lfp category $\A$ are precisely the filtered
  colimits of split subobjects of $A$ in the slice category $\A/A$.
\end{remark}

\begin{proposition}
  Let $\A$ be an lfp category in which all subobjects are pure. If $\A_{\fp} =
  \A_{\fg}$, then $\A$ is strictly lfp.
\end{proposition}
\begin{proof}
  Let $b\colon B\monoto A$ be a finitely generated subobject. Express it as a
  filtered colimit of split subobjects $b_i\colon B_i\monoto A$ (with $e_i\cdot
  b_i = \id_{B_i}$ for $e_i\colon A\to B_i$), $i\in I$, with the following colimit
  cocone in $\A/A$:
  \[
    \xymatrix{B_i\ar@<-2pt>@{>->}[rd]_{b_i} \ar[rr]^{c_i}&&B\ar@{>->}[ld]^{b}\\
      &A\ar@<-2pt>@{-->}[ul]_{e_i}
      &
      }
  \]
  Then in $\A$ we have expressed $B$ as a filtered colimit of the objects $B_i$
  with the cocone $(c_i)_{i\in I}$. It follows from our assumptions that $B$ is
  also finitely presentable, and therefore $\A(B,-)$ preserves that colimit.
  Hence, some $c_i$ is invertible (being both monic, due to $b_i=b\cdot c_i$,
  and split epic). Consequently, $B_i$ is finitely presentable. The finitary
  endomorphism $f=b_i\cdot e_i$ fixes the subobject $b$, as desired:
  \[
    f\cdot b = (b_i\cdot e_i) \cdot (b_i\cdot c_i^{-1})
    = b_i\cdot c_i^{-1}= b.
    \tag*{\endproofbox} 
  \]
  \def\endproof{}
\end{proof}
\begin{example}
  The following categories are strictly lfp because they satisfy all the
  assumptions
  of the above proposition. By a variety we mean an equational class of finitary
  (one-sorted) algebras.
  \begin{enumerate}
  \item\label{i:varietyFpFg} A variety $\A$ of algebras with $\A_{\fp} =
    \A_{\fg}$ in which every finitely generated subobject of a finitely
    generated object splits. By \cite[Theorem~2.1]{borceuxRosicky} all monomorphisms are
    pure.

    An example of such a variety are boolean algebras. Here $\A_{\fg} =
    \A_{\fp}$ are precisely the finite algebras. Since every epimorphism in
    $\Set_{\fp}$ splits, by Stone's Duality every monomorphism between finite
    boolean algebras splits.

  \item\label{i:RModNoether} \rmod{} for all regular, left-Noetherian rings $R$.
    Recall
    that $R$ is left-Noetherian if every left ideal $I\subseteq R$ is
    finitely generated; this implies that finitely generated left modules are
    finitely presentable \cite[Example 3.8.28]{Rowen88}.
    Recall further that regularity (in von Neumann's sense) means that for every $a\in
    R$ there exists $\bar a\in R$ with $a=a\cdot \bar a \cdot a$. For
    left-Noetherian rings, this condition is
    equivalent to \rmod{} having all monomorphisms pure, see \cite[Proposition 2.11.20]{Rowen88}.

    Regular rings are a wider class than semi-simple rings, so in the realm of
    left-Noetherian rings we have a simplification of the argument of
    Example~\ref{E:strict}\ref{e:Mod}.

  \item A special case of \ref{i:varietyFpFg}, which is the `non-abelian
    generalization' of \ref{i:RModNoether}, are varieties $\A$ with
    $\A_\fp=\A_\fg$ such that for every morphism $a\colon X\to Y$ of $\A_{\fp}$
    there exists $\bar a\colon Y\to X$ with $a=a\cdot \bar a\cdot a$, See
    \cite[Proposition 3.4]{borceuxRosicky}.

  \item $G$-modules over a field $K$, i.e.~the functor category
    \[
      (\kvec)^G,
    \]
    for a finite group $G$ and a field of characteristic $0$. (More generally:
    every field whose characteristic does not divide $|G|$.)

    By the classical Maschke's Theorem \cite[Theorem XIII.1.1]{Lang65} for every
    subobject $b\colon B\monoto A$ there exists a coproduct $A=B+C$ with $b$ as
    the left injection. Thus $b$ splits: consider $[\id_B,0]\colon A\to B$.
    Hence all monomorphisms are pure.

    The forgetful functor to $\kvec$ preserves colimits (computed object-wise).
    The free $G$-module $\actualphi n$ on $n$ generators thus has finite dimension
    (of the underling vector space). Indeed, $\actualphi 1$ has dimension $|G|$
    because its underlying space is spanned by $G$, see XIII, Section 1 of
    \cite{Lang65}.
    Hence $\actualphi n = \actualphi 1 + \cdots + \actualphi 1$ has dimension
    $n\cdot |G|$.

    It follows that every finitely generated $G$-module is finitely presentable.
    Indeed, it is a quotient of $\actualphi n$ for some $n$, thus, it is
    finite-dimensional. And every finite-dimensional $G$-module $A$ is finitely
    presentable in $(\kvec)^G$. This follows easily from $A$ being finitely
    presentable in $\kvec$, since the group action $G\times A\to A$ is
    determined by its domain restriction to the finite set $G\times X$, where
    $X$ is a base of $A$.
  \end{enumerate}
\end{example}

\begin{examples}\label{E:nonsemistrict} Here we present
  lfp categories $\A$ which are not semi-strictly lfp.
  For that it would be sufficient to exhibit an object $A$ such that no
  endomorphism is finitary. However, we also provide something stronger:
  In each case we present a non-finitary \emph{endofunctor} that is finitely
  bounded.
  \begin{enumerate}
  \item The category $\Un$. In Example~\ref{E:unary} we have already shown
    the promised endofunctor. Thus
    $\Un$ is not semi-strictly lfp. For the algebra $A = \coprod_{p} C_p$, where $p$
    ranges over all prime numbers, there exists no finitary endomorphism. 
    
  \item The category $\Int$-$\Set$ (of actions of the integers on sets). Since
    this category is equivalent to that of unary algebras with one invertible
    operation, the argument is as in~(1).

  \item The category $\Gra$ of graphs and their homomorphisms is not
    semi-strictly lfp (see
    Example~\ref{E:set}\ref{E:strictNotSemi}).
    
    Analogously to Example~\ref{E:unary} define an endofunctor
    $F$ on $\Gra$ by
    \[
      FX = 
      \begin{cases}
        \mathds{1} + X & \text{if $X$ contains no cycle and no infinite path}\\
        \mathds{1} & \text{else},
      \end{cases}
    \]
    where $\mathds{1}$ is the terminal object, and $Ff = \id_\mathds{1} + f$ if the
    codomain $X$ of $f$ fulfils $FX = \mathds{1} + X$. This functor is clearly
    finitely bounded, but for the graph $A$ consisting of a single
    infinite path, it does not preserve the
    colimit $A = \colim D_A$ of Remark~\ref{R:prelim}(2).
  \item $\Set^\Nat$. If $\mathds{1}$ is the terminal object, then
    $\Set^\Nat(\mathds{1}, B') = \emptyset$ for all finitely presentable
    objects $B$. We define $F$ on $\Set^\Nat$ by $FX = \mathds{1} + X$ if $X$
    has only finitely many non-empty components, and $FX = \mathds{1}$ else. 
  \end{enumerate}
\end{examples}
\begin{openproblem}
  Is the category $\Pos$ of posets strictly lfp? Is every finitely bounded
  endofunctor on $\Pos$ finitary?
\end{openproblem}
We next present two examples of rather important categories for which we prove
that they are not semi-strictly lfp either.
\begin{example}
  \label{ex:Nom}
  Nominal sets are not semi-strictly lfp. Let us first recall the definition of
  the category $\Nom$ of nominal sets (see e.g.~\cite{Pitts13}). We fix a
  countably infinite set $\V$ of \emph{atomic names}. Let $\perms(\V)$ denote
  the group of all finite permutations on $\V$ (generated by all
  transpositions). Consider a set $X$ with an action of this group, denoted by
  $\pi \cdot x$ for a finite permutation $\pi$ and $x \in X$. A subset $A
  \subseteq \V$ is called a \emph{support} of an element $x \in X$ provided that
  every permutation $\pi \in \perms(\V)$ that fixes all elements of $A$ also
  fixes $x$:
  \[
    \text{$\pi(a) = a$ for all $a \in A \implies \pi \cdot x = x$}. 
  \]
  A \emph{nominal set} is a set with an action of the group
  $\perms(\V)$ where every element has a finite support. The category
  $\Nom$ is formed by nominal sets and \emph{equivariant maps},
  i.e.~maps preserving the given group action. $\Nom$ is a
  Grothendieck topos, it is an lfp category (see
  e.g.~Pitts~\cite[Remark~5.17]{Pitts13}), and, as shown by
  Petri\c{s}an~\cite[Proposition~2.3.7]{petrisanphd}, the finitely
  presentable nominal sets are precisely those with finitely many
  orbits (where an orbit of $x$ is the set of all $\pi \cdot x$).

  It is a standard result that every element $x$ of a nominal set has the least
  support, denoted by $\supp(x)$. In fact, $\supp\colon X\to \powf(\V)$ is itself an
  equivariant map, where $\powf(\V)$ is the set of all finite subsets of $\V$ with the action given by  $\pi \cdot Y = \{ \pi(v) \mid v \in
  Y\}$. 
  Consequently, any two elements of the same orbit $x_1$ and $x_2 = \pi\cdot x_1$ have a
  support of the same size. In addition, if $f\colon  X \to Y$ is an equivariant map, it is clear that
  \begin{equation}\label{eq-supp}{\supp (f(x)) \subseteq \supp(x),} \quad\text{for every
  $x \in X$}.
  \end{equation}

  Now we present a non-finitary endofunctor on $\Nom$ which is finitely bounded.
  Consider for every natural number $n$ the nominal set $P_n = \{ Y \subseteq \V
  \mid |Y| = n\}$ with the nominal structure given
  element-wise, as for $\powf(\V)$ above. Clearly, $\supp(Y) = Y$ for every $Y \in P_n$.
  
  For $A = \coprod_{0 < n < \omega} P_n$ the existence of a finitary
  endomorphism leads to a contradiction. In fact, let the
  corresponding pair of morphisms
  $\xymatrix@1{A\ar@<2pt>[r]^f&X\ar@<2pt>[l]^{g}}$ with $X$
  orbit-finite be given. It is clear that, for every
  $x\in X$, $\supp(x)\not=\emptyset$, otherwise, by \eqref{eq-supp},
  we would have $\supp(g(x))=\emptyset$, which contradicts the fact
  that $\supp(Y)=Y\not=\emptyset$ for all $Y\in A$. We show below that for
  every $Y\in A$, $\supp(f(Y))=\supp(Y)=Y$, thus $X$ admits infinitely
  many cardinalities for $\supp(x)$ with $x\in X$, contradicting the
  orbit-finiteness of $X$.

  By~\eqref{eq-supp}, it remains to prove that
  $\supp(Y) \subseteq \supp(f(Y))$. To see this, fix an element $v$ of
  $\supp(f(Y))$, which is already known to be nonempty. Now for any given 
  element $w$ of $\supp(Y) = Y$, the equivariance of $f$ applied to 
  the transposition $\pi$ of $v$ and $w$ implies that
 \[
    w \in \pi \cdot \supp(f(Y)) = \supp(\pi \cdot f(Y)) = \supp(f(\pi \cdot Y)) = \supp(f(Y)).
  \]

This proves that $\Nom$ is not semi-strictly lfp.
   
  Analogously to Example~\ref{E:unary} we define an endofunctor $F$ on
  $\Nom$ by
  \[
    FX = 
    \begin{cases}
      \mathds{1} + X & \text{if $\Nom(P_n,X) = \emptyset$ for some $n <
        \omega$}
      \\
      \mathds{1} & \text{else.}
    \end{cases}
  \]
  For an equivariant map $f\colon  X \to Y$, if $FY = \mathds{1} + Y$, then also $FX =
  \mathds{1} + X$: given $\Nom(P_n, Y) = \emptyset$ for some $n$, then also
  $\Nom(P_n, X) = \emptyset$ holds. In that case put
  $Ff=\id_{ \mathds{1}}+f$ and else $Ff$ is the unique equivariant map to $FY =
  \mathds{1}$. A very similar argument as in Example~\ref{E:unary} shows that
  $F$ is finitely bounded. However, $F$ is not finitary, as it does not preserve
  the colimit $\coprod_{n<\omega}P_n$ of the chain $P_1 \subto P_1 + P_2 \subto
  P_1 + P_2 + P_3 \subto \cdots$.
\end{example}

We prove next that in the category $[\Set,\Set]_\fin$
of finitary set functors (known to be lfp \cite[Theorem 1.46]{AdamekR94}) finitely generated objects
coincide with the finitely presentable ones, yet this category fails to be
semi-strictly lfp.

\begin{remark} \label{R:quot}
  Recall that a quotient of an object $F$ of $[\Set,\Set]_\fin$ is represented
  by a natural transformation $\varepsilon\colon F\to G$ with epic components.
  Equivalently, $G$ is isomorphic to $F$ modulo a \emph{congruence} $\sim$. This
  is a collection of equivalence relations $\sim_X$ on $FX$ ($X\in \Set$) such
  that for every function $f\colon X\to Y$ given $p_1\sim_X p_2$  in $FX$, it
  follows that $Ff(p_1) \sim_Y Ff(p_2)$.
\end{remark}
We are going to characterize finitely presentable objects of $[\Set,\Set]_\fin$
as the super-finitary functors introduced in \cite{Adamek:1990:AAC:575450}:
\begin{definition} \label{D:fin}
  A set functor $F$ is called \emph{super-finitary} if there exists a natural
  number $n$ such that $Fn$ is finite and for every set $X$, the maps
  $Ff$ for $f\colon n\to X$ are jointly surjective, i.e.~they fulfil
  $FX=\bigcup_{f\colon n\to X} Ff[Fn]$.
\end{definition}
\begin{examples}\label{E:supfin}
  \begin{enumerate}
  \item \label{supfin:automata}
    The functors $A\times \Id^n$ are super-finitary for all finite sets $A$ and
    all $n \in \Nat$.
  \item \label{supfin:sig} 
  More generally, let $\Sigma$ be a finitary signature, i.e.~a set of
    operation symbols $\sigma$ of finite arities $|\sigma|$. The corresponding
    \emph{polynomial set functor}
    \[
      H_\Sigma X= \coprod_{\sigma\in \Sigma} X^{|\sigma|}
    \]
    is super-finitary iff the signature has only finitely many symbols.
    We call such signatures \emph{super-finitary}.
  \item \label{superfin:sub}
    Every subfunctor $F$ of $\Set(n,-)$, $n\in \Nat$, is super-finitary.
      Indeed, assuming $FX\subseteq \Set(n,X)$ for all $X$, we are to find, for
      each $p\colon n\to X$ in $FX$, a member $q\colon n\to n$ of $Fn$ with
      $p=Ff(q)$ for some $f\colon n\to X$. That is, with $p=f\cdot q$.
      Choose a function $g\colon X\to n$ monic on $p[n]$. Then there exists
      $f\colon n\to X$ with $p=f\cdot g\cdot p$. From $p\in FX$ we deduce
      $Fg(p)\in Fn$, that is, $g\cdot p\in Fn$. Thus $q=g\cdot p$ is the desired
      element: we have $p=f\cdot q=Ff(q)$.
    \item \label{supfin:quotients} Every quotient $\varepsilon\colon F\twoheadrightarrow G$ of a
      super-finitary functor $F$ is super-finitary. Indeed, given $p\in GX$,
      find $p'\in FX$ with $p=\varepsilon_X(p')$. There exists $q'\in
      Fn$ with $p'=Ff(q')$ for some $f\colon n\to X$. We conclude
      that $q=\varepsilon_n(q')$ fulfils $p=Gf(q)$ from the naturality of $\varepsilon$.
  \end{enumerate}
\end{examples}

\begin{lemma} \label{L:fin}
  The following conditions are equivalent for every set functor $F$:
  \begin{enumerate}
    \item $F$ is super-finitary \label{L:fin:supfin}
    \item $F$ is a quotient of the polynomial functor $H_\Sigma$ for a
      super-finitary signature $\Sigma$, and \label{L:fin:poly}
    \item $F$ is a quotient of a functor $A\times \Id^n$ for $A$ finite and $n\in \Nat$. \label{L:fin:quot}
  \end{enumerate}
\end{lemma}
\begin{proof}
  \ref{L:fin:quot}$\implies$\ref{L:fin:poly} is clear and for \ref{L:fin:poly}$\implies$\ref{L:fin:supfin}
  see the Examples  \ref{supfin:sig} and \ref{supfin:quotients} above. To
  prove \ref{L:fin:supfin}$\implies$\ref{L:fin:quot}, let $F$ be super-finitary and
  put $A=Fn$ in the above definition. Apply Yoneda Lemma to $\Id^n\cong
  \Set(n,-)$ and use that $[\Set,\Set]_\fin$ is cartesian closed:
 \[
   \frac{
     Fn \xrightarrow{~~\cong~~} [\Set,\Set]_\fin(\Set(n,-), F)
   }{
     \varepsilon\colon Fn \times \Set(n,-)\xrightarrow{~~\phantom{\cong}~~} F
   }
 \]
 The definition of super-finitary shows that $\varepsilon_X$ is surjective for
 every $X$.
\end{proof}

\begin{proposition} \label{P:fin}
  Super-finitary functors are closed in $[\Set,\Set]_\fin$ under finite products,
  finite coproducts, subfunctors, and hence under finite limits.
\end{proposition}
\begin{proof}
  \begin{enumerate}
    \item Finite products and coproducts are clear: given quotients
      $\varepsilon_{i}\colon A_i\times \Id^{n_i}\twoheadrightarrow F_i$, $i\in \{1,2\}$, then
      $F_1\times F_2$ is super-finitary due to the quotient
      \[
        \varepsilon_1 \times \varepsilon_2\colon (A_1\times A_2)\times \Id^{n_1+n_2}
        \to F_1\times F_2.
      \]
      Suppose $n_1\ge n_2$, then we can choose a quotient $\varphi\colon
      A_2\times \Id^{n_1} \twoheadrightarrow A_2\times \Id^{n_2}$.
      This proves that $F_1+F_2$ is super-finitary due to the quotient
      \[
        \varepsilon_1+(\varepsilon_2\cdot \varphi)\colon
        (A_1+A_2)\times \Id^{n_1}
        \cong A_1\times \Id^{n_1}+A_2\times \Id^{n_1}
        \to F_1+F_2.
      \]
    \item Let $\mu\colon G\rightarrowtail F$ be a subfunctor of a super-finitary
      functor $F$ with a quotient $\varepsilon\colon A\times
      \Id^n\twoheadrightarrow F$. Form a pullback (object-wise in $\Set$) of
      $\varepsilon$ and $\mu$:
      \[
        \begin{tikzcd}
          H
          \arrow[>->]{r}{\bar \mu}
          \arrow[->>]{d}[swap]{\bar \varepsilon}
          \pullbackangle{-45}
          &
          A \times \Id^n
          \arrow[->>]{d}{\varepsilon}
          \\
          G
          \arrow[>->]{r}{\mu} & F
        \end{tikzcd}
      \]
      For each $a\in A$, the preimage $H_a$ of $\{a\}\times \Id^n\cong
      \Set(n,-)$ under $\bar\mu$ is super-finitary by Example \ref{superfin:sub}
      above. Since $A\times \Id^n = \coprod_{a\in A} \{a\}\times \Id^n$ and
      preimages under $\bar \mu$ preserve coproducts, we have $H=\coprod_{a\in
        A} H_a$ and so $G$ is a quotient of the super-finitary functor~$H$.
  \end{enumerate}
\end{proof}
\begin{lemma} \label{L:kernel2fpfg} Let $\C$ be an lfp
  category with finitely generated objects closed under kernel pairs and in
  which strong epimorphisms are regular. Then finitely presentable
  and finitely generated objects coincide.
\end{lemma}
\begin{proof} We apply Remark \ref{R:prelim}\ref{I:fingen}:
  Consider a strong epimorphism $c\colon X\twoheadrightarrow Y$ with $X$
  finitely presentable. We are to show that $Y$ is finitely presentable. Let
  $p,q\colon K\rightrightarrows X$ be the kernel pair of $c$, then $K$ is
  finitely generated. Hence there is some finitely presentable object $K'$ and a
  strong epimorphism $e\colon K'\twoheadrightarrow K$:
  \[
    \xymatrix{
      K'
      \ar@{->>}[r]^-{e}
      & K
      \ar@<2pt>[r]^-{p}
      \ar@<-2pt>[r]_-{q}
      & X
      \ar@{->>}[r]^-{c}
      & Y
    }
  \]
  Since the strong epimorphism $c$ is also regular, it is the coequalizer of its
  kernel pair $(p,q)$; furthermore $e$ is epic, thus $c$ is also the coequalizer of
  $p\cdot e$ and $q\cdot e$. This means that $Y$ is a finite colimit of finitely
  presentable objects and thus it is finitely presentable.
\end{proof}
\begin{corollary}
  \label{cor:setFunNotSStrict}
  $[\Set,\Set]_\fin$ is not semi-strictly lfp.
\end{corollary}
\begin{proof}
  We use the subfunctors
  \[
    \bar \pow \subseteq \pow_0 \subseteq \pow
  \]
  of the power-set functor $\pow$ given by $\pow_0X=\pow X\setminus\{\emptyset\}$ and $\bar \pow X
  =\{M\in \pow_0 X\mid M\text{ finite}\}$. Then $\bar\pow$ is an object of
  $[\Set,\Set]_\fin$ which is clearly not super-finitary. The only
  endomorphism of $\bar \pow$ is $\id_{\bar \pow}$. Indeed for $\pow_0$ this has
  been proven in \cite[Proposition 5.4]{adamekSousa}; the same proof applies to
  $\bar \pow$.
  And $\id_{\bar \pow}$ is not finitary: otherwise $\bar \pow$ would be a
  quotient of a finitely presentable object, thus, it would be super-finitary
  (due to Lemma~\ref{L:fin}).
\end{proof}

\begin{corollary} \label{cor:supfinfgfp}
  For a finitary set functor, as an object of $[\Set,\Set]_\fin$, the following
  conditions are equivalent:
  \begin{enumerate}
  \item \label{supfin:fp} finitely presentable,
  \item \label{supfin:fg} finitely generated, and
  \item \label{supfin:supfin} super-finitary.
  \end{enumerate}
\end{corollary}
\begin{proof}
  To verify
  \ref{supfin:fg}$\implies$\ref{supfin:supfin}, let $F$ be finitely generated.
  For every finite subset $A\subseteq Fn$, $n\in \Nat$, we have a subfunctor
  $F_{n,A}\subseteq F$ given by
  \[
    F_{n,A} X = \,\bigcup_{\mathclap{f\colon n\to X}}\, Ff[A].
  \]
  Since $F$ is finitary, it is a directed union of all these
  subfunctors. This implies $F\cong
  F_{n,A}$ for some $n$ and $A$, and $F_{n,A}$ is clearly super-finitary.
  
  For \ref{supfin:supfin}$\implies$\ref{supfin:fg}, combine Lemma \ref{L:fin} and Example \ref{E:supfin}\ref{supfin:automata}.

   \ref{supfin:fp}$\iff$\ref{supfin:fg} follows by Lemma~\ref{L:kernel2fpfg}. 
\end{proof}

\section{\texorpdfstring{$\lambda$}{Lambda}-Accessible Functors}

Almost everything we have proved above generalizes to locally
$\lambda$-presentable categories for every infinite regular cardinal
$\lambda$. Recall that an object $A$ of a category $\A$ is
\emph{$\lambda$-presentable} (\emph{$\lambda$-generated}) if its
hom-functor $\A(A,-)$ preserves $\lambda$-filtered colimits (of
monomorphisms). A category $\A$ is \emph{locally
  $\lambda$-presentable} if it is cocomplete and has a set of
$\lambda$-presentable objects whose closure under $\lambda$-filtered
colimits is all of $\A$. Functors preserving $\lambda$-filtered
colimits are called \emph{$\lambda$-accessible}. We denote by $\A_\lp$ and
$\A_\lg$  full subcategories representing (up to isomorphism) all $\lambda$-presentable and $\lambda$-generated
objects, respectively.

All of Remark~\ref{R:prelim} holds for $\lambda$ in lieu of
$\aleph_0$, with the same references in~\cite{AdamekR94}.

If $\lambda = \aleph_1$ we speak about {\em locally countably presentable
categories}, {\em countably presentable objects}, etc.

\begin{examples}
  \begin{enumerate}
  \item Complete metric spaces. We denote by
    \[
      \CMS
    \]
    the category of complete metric spaces of diameter $\leq 1$ and
    non-expanding functions, i.e.~functions $f\colon  X \to Y$ such that for
    all $x,y \in X$ we have $d_Y(f(x),f(y)) \leq d_X(x,y)$. This
    category is locally countably presentable. The classes of
    countably presentable and countably generated objects coincide and
    these are precisely the compact spaces.

    Indeed, every compact (= separable) complete metric space is countably
    presentable, see \cite[Corollaries~2.9]{ammu15}. And every countably
    generated space $A$ in $\CMS$ is separable: consider the countably filtered
    diagram of all spaces $\bar X\subseteq A$ where $X$ ranges over countable
    subsets of $A$ and $\bar X$ is the closure in $A$. Since $A$ is the colimit
    of this diagram, $\id_A$ factorizes through one of the embeddings $\bar
    X\hookrightarrow A$, i.e.~$A=\bar X$ is separable.
  \item Complete partial orders. Denote by
    \[
      \CPO
    \]
    the category of $\omega$-cpos, i.e.~of posets with joins of
    $\omega$-chains and monotone functions preserving joins of
    $\omega$-chains. This is also a locally countably presentable
    category. An $\omega$-cpo is countably presentable (equivalently,
    countably generated) iff it has a countable subset which is dense
    w.r.t.~joins of $\omega$-chains.
  \end{enumerate}
\end{examples}

Following our convention in Section~\ref{sec:fin} we speak about a
\emph{$\lambda$-generated subobject} $m\colon  M \monoto A$ of $A$ if $M$ is
a $\lambda$-generated object of $\A$. This leads to a generalization
of the notion of finitely bounded functors to $\lambda$-bounded
ones. The latter terminology stems from Kawahara and
Mori~\cite{KawaharaM00}, where endofunctors on sets were
considered. Our terminology is slightly different in that
$\lambda$-generated subobjects in $\Set$ have cardinality less than
$\lambda$, whereas subsets of cardinality less than or equal to $\lambda$
were considered in loc.~cit.

\begin{definition}
  A functor $F\colon  \A \to \B$ is called \emph{$\lambda$-bounded} provided
  that given an object $A$ of $\A$, every $\lambda$-generated
  subobject  $m_0\colon  M_0 \monoto FA$ in $\B$ factorizes through the $F$-image of a 
  $\lambda$-generated subobject $m\colon  M \monoto A$ in $\A$:
  \[
    \xymatrix{
      & FM\ar[d]^{Fm}\\
      M_0\ar@{-->}[ru]\ar@{ >->}[r]_{m_0} & FA
    }
  \]
\end{definition}
\begin{theorem}\label{T:bound}
  Let $\A$ be a locally $\lambda$-presentable category in which every
  $\lambda$-generated object is $\lambda$-presentable. Then for all
  functors from $\A$ to locally $\lambda$-presentable
  categories preserving monomorphisms we have the equivalence
  \[
    \text{$\lambda$-accessible} \iff \text{$\lambda$-bounded}.
  \]
\end{theorem}
The proof is completely analogous to that of Theorem~\ref{T:finbound}. 

\begin{example}
  The Hausdorff endofunctor $\H$ on $\CMS$ was proved to be accessible (for some
  $\lambda$) by van Breugel et al.~\cite{vanBreugelEA07}. Later it was shown to
  be even finitary~\cite{ammu15}. However, these proofs are a bit involved.
  Using Theorem~\ref{T:bound} we provide an easy argument why the Hausdorff
  functor is countably accessible. (Which, since $\CMS$ is not lfp but is
  locally countably presentable, seems to be the `natural' property.)

  Recall that for a given metric space $(X,d)$ the distance of a point $x\in X$
  to a subset $M \subseteq X$ is defined by $d(x, M) = \inf_{y \in M} d(x,y)$.
  The \emph{Hausdorff distance} of subsets $M, N \subseteq X$ is defined as the
  maximum of $\sup_{x \in M} d(x,N)$ and $\sup_{y\in N} d(y, M)$. The
  \emph{Hausdorff functor} assigns to every complete metric space $X$ the space
  $\H X$ of all non-empty compact subsets of $X$ equipped with the Hausdorff
  metric. It is defined on  non-expanding maps by taking the direct images.
  We now easily see that $\H$ is countably accessible:
  \begin{enumerate}
  \item $\H$ preserves monomorphisms. Indeed, given $f\colon  X \monoto Y$
    monic, then $f[M] \neq f[N]$ for every pair $M, N$ of distinct elements of
    $\H X$, thus $\H f$ is monic, too.
  \item $\H$ is countably bounded. In order to see this, let $m_0\colon  M_0
    \subto \H X$ be a subspace with $M_0$ compact, and choose a
    countable dense subset $S \subseteq M_0$. For every element $s \in
    S$ the set $m_0(s) \subseteq X$ is compact, hence, separable; choose
    a countable dense set $T_s \subseteq m_0(s)$. For the countable set
    $T = \bigcup_{s \in S} T_s$ form the closure in $X$ and denote it
    by $m\colon  M \subto X$. Then $M$ is countably generated, and $M_0
    \subseteq \H m[\H M]$; indeed, for every $x \in M_0$ we have $m_0(x)
    \subseteq M$ because $M$ is closed, and this holds whenever $x \in
    S$ (due to $m_0(x) = \overline{T_x}$).
  \end{enumerate}
\end{example}
In the following definition a morphism is called \emph{$\lambda$-ary} if it factorizes
through a $\lambda$-presentable object.
\takeout{
\begin{definition}
  A locally $\lambda$-presentable category $\A$ is called {\em strictly} or {\em
    semi-strictly} locally $\lambda$-presentable provided that every morphism
  $b\colon B\to A$ in $\A_{\lambda}/A$ factorizes through a morphism $b'\colon B'\to A$ in
  $\A_{\lambda}/A$ for which some $f\colon A\to B'$ exists and, in the case of strict
  locally $\lambda$-presentable, $f\cdot b$ is such a factor,
  i.e.~$b=b'\cdot(f\cdot b)$.
  \[
    \begin{array}{c}
      \xymatrix{B\ar[rd]_b\ar@{-->}[rr]&&B'\ar@<2pt>[ld]^{b'}\\
        &A\ar@<2pt>[ur]^f&}
      \\[5pt]
      \text{\em semi-strictly locally $\lambda$-presentable}
    \end{array}
    \qquad \qquad 
    \begin{array}{c}
      \xymatrix{
        B\ar[rd]_b\ar[rr]^{f\cdot b}
        &&
        B'\ar@<2pt>[ld]^{b'}\\
        &
          A\ar@<2pt>[ur]^f&}
        \\[5pt]
      \text{\em strictly locally $\lambda$-presentable}
    \end{array}
  \]
  
\end{definition}
}
\begin{definition}
  A locally $\lambda$-presentable category is called
  \begin{enumerate}
  \item \emph{semi-strictly locally $\lambda$-presentable} if every object has a $\lambda$-ary endomorphism;
  \item \emph{strictly locally $\lambda$-presentable} if every object has, for each $\lambda$-generated
    subobject $m$, a finitary endomorphism $u$ fixing that subobject
    (i.e.~$u\cdot m = m$).
  \end{enumerate}
\end{definition}
Observe that Remark \ref{R:strict}  immediately generalizes to an
 	arbitrary $\lambda$.
\begin{examples}
\begin{enumerate}
\item $\Set^S$ is strictly locally $\lambda$-presentable iff $\card S
  < \lambda$. This is analogous to Example~\ref{E:strict}(1).

\item The category $\mathsf{Grp}$ of groups is semi-strictly locally
  $\lambda$-presentable by the same argument as in
  Example~\ref{E:set}\ref{R:strict:zeroobject}. However, $\mathsf{Grp}$ is not
  strictly locally $\lambda$-presentable for any infinite
  cardinal $\lambda$.

  To see this, let $A$ be a simple group of cardinality at least $\lambda^\lambda$.
  (Recall that for every set $X$ of cardinality $\geq 5$ the group of even
  permutations on $X$ is simple.) Since $\mathsf{Grp}$ is an lfp category, there exists
  a non-zero homomorphism $b\colon  B \to A$ with $B$ finitely presentable. Given a
  commutative diagram
  \[
    \vcenter{
      \xymatrix{
        B \ar[rr]^-{f \cdot b} \ar[rd]_b && B' \ar[ld]^{b'}\\
        & A
        }
      }
      \qquad
      \text{for some $f\colon  A \to B'$}
  \]
  we show that $B'$ is not $\lambda$-presentable. Indeed, since $b$ is non-zero,
  we see that so is $f\colon  A \to B'$. Since $A$ is simple, $f$ is monic, hence
  $\card B' \geq \lambda^\lambda$. However, every $\lambda$-presentable group
  has cardinality at most $\lambda$. Thus, by an argument analogous to
  Remark~\ref{R:strict}\ref{R:strict:str}, $\mathsf{Grp}$ is not strictly
  locally $\lambda$-presentable.

  \takeout{
  The category of groups is not semi-strictly\smnote{I believe
    this should be `semi-strictly'.} \lsnote{I agree with Stefan: the category of groups should be semi-strictly locally
  $\lambda$-presentable, because it is easy to see that the initial object  is always  $\lambda$-presentable, for any $\lambda$, thus we may prove semi-strictness exactly as in \ref{E:set}\ref{R:strict:zeroobject}.}locally
  $\lambda$-presentable for any infinite cardinal $\lambda$.

  Indeed, let $A$ be a simple group of cardinality at least $\lambda^\lambda$.
  (Recall that for every set $X$ of cardinality $\geq 5$ the group of even
  permutations on $X$ is simple.)
  Then every endomorphism of $A$ is monic, hence its image has cardinality at
  least $\lambda^\lambda$, too. However,
  every $\lambda$-presentable group has cardinality at most $\lambda$.
  Thus, $A$ has no $\lambda$-ary endomorphism.}

\item The category $\Nom$ of nominal sets is strictly locally countably
  presentable. In order to prove this, we first verify that countably
  presentable objects are precisely the countable nominal
  sets.
  \begin{enumerate}
  \item
  Let $X$ be a countably presentable nominal set. Then every countable
  choice of orbits of $X$ yields a countable subobject of $X$ in
  $\Nom$. Thus $X$ is a countably directed union of countable
  subobjects. Since $X$ is countably presentable, it follows that $X$
  is isomorphic to one of these subobjects. Thus, $X$ is countable.

  \item
  Conversely, every countable nominal set is countably presentable
  since countably filtered colimits of nominal sets are formed on the
  level of sets (i.e.~these colimits are preserved and reflected by
  the forgetful functor $\Nom \to \Set$).
  \end{enumerate}

  Now let $b\colon B \to A$ be a morphism in $\Nom$ with $B$
  countable. We have $A = \Im(b) + C$ for some subobject $C$
  of $A$. Indeed, every nominal set is a coproduct of its orbits, and the
  equivariance of $b$ implies that $\Im(b)$ is a coproduct of some of
  the orbits of $A$. Furthermore, let $m\colon C_1 \monoto C$ be a
  subobject obtained by choosing one orbit from each isomorphism class
  of orbits of $C$. We obtain a surjective equivariant map
  $e\colon C \epito C_1$ by choosing, for every orbit in
  $C \setminus C_1$, a concrete isomorphism to an orbit of $C_1$ and
  for every $x \in C_1 \subseteq C$ putting $e(x) = x$. Then we have
  $e \cdot m = \id_{C_1}$, i.e.~$m$ is a split monomorphism of
  $\Nom$. In the appendix we prove that there are (up to isomorphism)
  only countably many single-orbit nominal sets. Hence, $C_1$ is
  countable, and thus so is $B' = \Im(b) + C_1$. Moreover, the
  morphisms $b' = \id + m\colon B' \to A$ and
  $f\colon \id + e\colon A \to B'$ clearly satisfy the desired
  property $b = b'\cdot f \cdot b$, see Remark~\ref{R:strict}\ref{R:strict:str}.
\end{enumerate}
\end{examples}
\begin{proposition}
  Every semi-simple locally presentable category is strictly locally
  $\lambda$-presentable for some $\lambda$.
\end{proposition}
\begin{proof} Let $\A$  be a locally $\kappa$-presentable category that is semi-simple.
  \begin{enumerate}
  \item $\A$ has only a set of simple objects up to
    isomorphism. Indeed, we have a set $\A_{\kappa}$ representing all
    $\kappa$-presentable objects. Given a simple object $A$, express
    it as a colimit of a $\kappa$-filtered diagram in $\A_{\kappa}$
    with a colimit cocone $c_i\colon C_i\to A$, $i\in I$. Since $\A$
    is locally presentable, it has (strong epi,
    mono)-factorizations~\cite[Proposition~1.61]{AdamekR94}. Then,
    since $A$ is simple, either it is a strong quotient of some $C_i$
    or it is an initial object. Thus, every simple object
    is a strong quotient of a $\kappa$-presentable one. The desired
    statement follows since every locally
    presentable category is cowellpowered~\cite[Theorem~1.58]{AdamekR94}.

  \item Let $\lambda\geq \kappa$ be a regular cardinal such that every
    semi-simple object is $\lambda$-presentable. Then $\A$ is locally
    $\lambda$-presentable, and the rest of the proof is completely
    analogous to point \ref{P:semi-simple:2} in the proof of
    Proposition~\ref{P:semi-simple}.
  \end{enumerate}
\end{proof}

\begin{corollary}
  For every semi-simple ring $R$ the category
  $R$-$\mathsf{Mod}$ is strictly locally $\lambda$-presentable
  provided that $\lambda > 2^{|R\times R|}$.
\end{corollary}
\noindent
Indeed, the module $R$ has less than $\lambda$ quotient modules. As in
Example~\ref{E:strict}\ref{e:Mod} each quotient is
$\lambda$-presentable in $R$-$\mathsf{Mod}$, and the rest is as in
that example.
\begin{corollary}
  Every atomic Grothendieck topos with a set of atoms
  (up to isomorphism) is strictly locally $\lambda$-presentable for
  some $\lambda$.
\end{corollary}
\noindent
Being a Grothendieck topos, our category is locally
$\lambda$-presentable for some $\lambda$. We can choose $\lambda$ to
be (a) larger than the number of atoms up to isomorphism and (b) such
that every atom is $\lambda$-presentable. Then our topos is strictly
locally $\lambda$-presentable.
\begin{example}\label{E:gpd-lambda}
  The category of presheaves on a small groupoid is strictly locally
  $\lambda$-presentable. Indeed, the proof that there is, up to
  isomorphism, only a set of atomic presheaves is analogous to
  Lemma~\ref{L:gpd}.
\end{example}
\begin{theorem}
  Let $\A$ be a  locally $\lambda$-presentable category.
  \begin{enumerate}
  \item  If $\A$ is strictly locally $\lambda$-presentable,  then
    for all functors from $\A$ to a locally $\lambda$-presentable
    category $\B$ with $\B_\lp= \B_\lg$ we have
  \[
    \text{$\lambda$-accessible} \iff \text{$\lambda$-bounded}.
  \]
  
\item Conversely, if this equivalence holds for all functors to $\Set$, then
  $\A$ is semi-strictly locally $\lambda$-presentable and $\A_\lp = \A_\lg$.
\end{enumerate}
\end{theorem}
The proofs are completely analogous to those of Theorems~\ref{T:boundstrict} and \ref{T:equiv2semstrict}.
%
%

\begin{remark} Assume that we work in a set theory distinguishing between sets and classes (e.g. Zermelo-Fraenkel theory) or distinguishing universes, so that by  `a class' we take a member of the next higher universe of that of all small sets. 
Then we form a super-large category 
$$\Class$$
of classes and class functions. It plays a central role in the paper of Aczel and Mendler \cite{AczelMen89} on terminal coalgebras. An endofunctor $F$ of  $\Class$ in that paper is called {\em set-based}
if for every class $X$ and every element $x\in FX$ there exists a subset
$i\colon Y\monoto X$ such that $x$ lies in $Fi[FX]$. This corresponds to $\infty$-bounded where $\infty$ stands for `being large'. The corresponding concept of $\infty$-accessibility is evident:
\end{remark}

\begin{definition} A diagram $D\colon \D\to \Class$, with $\D$ not necessarily small, is called {\em $\infty$-filtered} if every small subcategory of $\D$ 
has a cocone in $\D$. An endofunctor of $\Class$ is called {\em $\infty$-accessible} if it preserves colimits of $\infty$-filtered diagrams.
\end{definition}

\begin{proposition} An endofunctor of $\Class$ is set-based iff it is $\infty$-accessible.
\end{proposition}

\begin{proof}(1) For every morphism $b\colon B\to A$ in $\Class$ with $B$ small
  factorizes in $\Set/A$ through a morphism $b'\colon B'\to A$ in $\Set/A$ where the factorization $f$ fulfils $b=b'\cdot (f\cdot b)$.
(Shortly: $\Class$ is strictly locally $\infty$-presentable.) The proof is the same as that of Example \ref{E:set}(2).

(2) The rest is completely analogous to part (1) of the proof of Theorem \ref{T:boundstrict}
\end{proof}

\begin{remark} Assuming, moreover, that all proper classes are mutually bijective, it follows that {\em every} endofunctor on $\Class$ is $\infty$-accessible, see \cite{AMV04}.
\end{remark}

%
%
\bibliographystyle{myabbrv}
\bibliography{refs}

\newpage
\begin{appendix}
  \section{Details on  Single-Orbit Nominal Sets}
  In this appendix we prove that in the category $\Nom$ of nominal
  sets there are (up to isomorphism) only countably many nominal sets
  having only one orbit. To this end we consider the nominal sets
  $\V^{\#n}$ of injective maps from $n=\{0,1,\dots,\, n-1\}$ to $\V$.
  The group action on $\V^{\#n}$ is component-wise, in other words, it
  is given by postcomposition: for $t\colon n \monoto \V$ and
  $\pi \in \perms(\V)$ (which is a bijective map $\pi\colon \V \to \V$)
  the group action is the composed map
  $\pi \cdot t\colon n \monoto \V$.  Thus, for every $t:n\monoto \V$
  of $\V^{\#n}$, $\supp(t)=\{t(i)\mid i<n\}$.
  \begin{lemma}
    Up to isomorphism, there are only countably many single-orbit nominal sets.
  \end{lemma}
  \proof
    Every single-orbit nominal set $Q$ whose elements have supports of
    cardinality $n$ is a quotient of the (single-orbit) nominal set
    $\V^{\# n}$ (see~\cite[Exercise~5.1]{Pitts13}).  Indeed, if
    $Q=\{ \pi\cdot x \mid \pi \in \perms(\V)\}$ with
    $\supp(x)=\{a_0,\ldots,a_{n-1}\}$, let $t\colon n \monoto \V$ be
    the element of $\V^{\# n}$ with $t(i)=a_i$ and define
    $q\colon \V^{\# n} \epito Q$ as follows: for every
    $u\in \V^{\# n}$ it is clear that there is some
    $\pi \in \perms(\V)$ with $u=\pi \cdot t$; put $q(u)=\pi \cdot
    x$. This way, $q$ is well-defined (since
    $\supp(x)=\{t(i)\mid i<n\}$) and equivariant.

    For every $n\in \Nat$, the quotients of $\V^{\#n}$ are given by equivariant
    equivalence relations on $\V^{\#n}$. We prove that we have a bijective
    correspondence between the set of all quotients with $|\supp([t]_{\sim})| =
    n$ for all $t \in \V^{\#n}$ and the set of all subgroups of $\perms(n)$.
    \begin{enumerate}
    \item Given an equivariant equivalence $\sim$ on $\V^{\# n}$ put
      \[
        S = \{ \sigma \in \perms(n) \mid \forall (t\colon n\monoto
        \V)\colon t\cdot \sigma  \sim t\}.
      \]
      Note that since $\sim$ is equivariant (and composition of
      maps is associative), $\forall$ can equivalently be replaced by
      $\exists$:
      \[
        S = \{ \sigma \in \perms(n) \mid \exists (t\colon n\monoto
        \V)\colon t\cdot \sigma \sim t\}.
      \]
      It is easy to verify that $S$ is a subgroup of
      $\perms(n)$. Moreover, we have that, for every $t,u\in \V^{\#n}$,
      \begin{equation}\label{Aa}
        t\sim u \quad\iff\quad u=t\cdot \sigma\quad \text{for some $\sigma \in S$}.
      \end{equation}
      Indeed, ``$\Longleftarrow$'' is obvious. For ``$\Longrightarrow$'' suppose that
      $t\sim u$. Since $|\supp([t]_{\sim})|=n$, we have that
      $\supp(t)=\supp([t]_{\sim})=\supp([u]_{\sim})=\supp(u)$; thus,
      there is some $\sigma \in \perms(n)$ such that $u=t\cdot
      \sigma$. Consequently, $t \sim t\cdot \sigma$,
      showing that $\sigma \in S$.

      \enlargethispage{1pt}
    \item For every subgroup $S$ of $\perms(n)$, it is clear that the
      relation $\sim$ defined by \eqref{Aa} is an equivariant
      equivalence. We show that, moreover, $|\supp([t]_{\sim})|=n$ for
      every $t\in \V^{\#n}$. We have $|\supp([t]_\sim)| \le n$ because
      the canonical quotient map $[-]_\sim$ is equivariant. In order
      to see that $|\supp([t]_\sim)|$ is not smaller than $n$, assume
      $a \in \supp(t)\setminus \supp([t]_\sim)$ and take any element
      $b\not\in \supp(t)$. Then $(a\,b)\cdot [t]_\sim = [t]_\sim$,
      i.e.~there is some $\sigma\in \perms(n)$ with
      $(a\,b)\cdot t\cdot \sigma = t$, which is a contradiction to
      $b\not\in\supp(t) = \supp(t \cdot \sigma) = \{t(i) \mid i <
      n\}$.

    \item It remains to show that, given two subgroups $S$ and $S'$
      which determine the same equivariant equivalence relations
      $\sim$ via \eqref{Aa}, then $S=S'$. Indeed, given $\sigma\in S$,
      we have $t = (t \cdot \sigma) \cdot \sigma^{-1}$ and therefore
      $t\cdot \sigma \sim t$ for every $t\in \V^{\#n}$. By \eqref{Aa}
      applied to $S'$, this implies that
      $t =t\cdot\sigma \cdot \sigma'$ for some $\sigma' \in S'$. Since
      $t$ is monic, we obtain $\sigma \cdot \sigma'=\id_n$,
      i.e.~$\sigma=(\sigma')^{-1}\in S'$. This proves $S\subseteq S'$,
      and the reverse inclusion holds by symmetry.  \endproof
  \end{enumerate}
\end{appendix}
\end{document}